\documentclass{amsart}
\usepackage{amssymb}
\usepackage[all]{xy}
\usepackage{graphicx}
\usepackage[mathcal]{euscript}
\usepackage{verbatim}

\newcommand{\D}{\mathbb{D}}

\newcommand{\cA}{\mathcal{A}}
\newcommand{\cB}{\mathcal{B}}
\newcommand{\cC}{\mathcal{C}}
\newcommand{\cD}{\mathcal{D}}
\newcommand{\cE}{\mathcal{E}}
\newcommand{\cF}{\mathcal{F}}
\newcommand{\cG}{\mathcal{G}}
\newcommand{\cH}{\mathcal{H}}
\newcommand{\cM}{\mathcal{M}}
\newcommand{\cO}{\mathcal{O}}
\newcommand{\cP}{\mathcal{P}}
\newcommand{\cQ}{\mathcal{Q}}
\newcommand{\fS}{\mathfrak{S}}

\newcommand{\tC}{\tilde C}
\newcommand{\tP}{\tilde P}
\newcommand{\tU}{{\tilde U}}
\newcommand{\tX}{{\tilde X}}

\newcommand{\tpi}{\tilde \rho}
\newcommand{\ti}{\tilde\imath}
\newcommand{\tj}{\tilde\jmath}

\newcommand{\id}{\mathrm{id}}

\newcommand{\Z}{\mathbb{Z}}
\newcommand{\Qlb}{\bar{\mathbb{Q}}_l}

\newcommand{\p}{{}^p}
\newcommand{\pH}{{}^p\!H}
\newcommand{\std}{{}^{\mathrm{std}}}
\newcommand{\st}{\mathrm{std}}

\newcommand{\ctrc}{\tau}
\newcommand{\bete}{{\text{b\^ete}}}
\newcommand{\topl}{{G\text{-}\mathrm{gen}}}
\newcommand{\Coh}{\mathfrak{Coh}}

\newcommand{\IC}{\mathrm{IC}}
\newcommand{\cIC}{\mathcal{IC}}
\newcommand{\bIC}{{}^D\!\cIC}

\DeclareMathOperator{\bSpec}{\mathbf{Spec}}
\DeclareMathOperator{\Spec}{Spec}
\DeclareMathOperator{\depth}{depth}

\DeclareMathOperator{\Hom}{Hom}
\DeclareMathOperator{\cHom}{\mathcal{H}\mathit{om}}
\DeclareMathOperator{\cRHom}{\mathit{R}\mathcal{H}\mathit{om}}
\DeclareMathOperator{\real}{real}
\DeclareMathOperator{\tchar}{char}
\DeclareMathOperator{\supp}{supp}

\DeclareMathOperator{\dc}{codim}

\newtheorem{thm}{Theorem}[section]
\newtheorem{lem}[thm]{Lemma}
\newtheorem{prop}[thm]{Proposition}
\newtheorem{cor}[thm]{Corollary}
\newtheorem{conj}[thm]{Conjecture}
\theoremstyle{definition}
\newtheorem{defn}[thm]{Definition}
\newtheorem*{conv}{Notational Convention}
\theoremstyle{remark}
\newtheorem{rmk}[thm]{Remark}
\newtheorem{exam}[thm]{Example}

\title[Perverse coherent sheaves and the geometry
of
special pieces]{Perverse coherent sheaves and the geometry
of
special pieces in the unipotent variety}
\author{Pramod N.~Achar}
\address{Department of Mathematics\\
  Louisiana State University\\
  Baton Rouge, LA 70803}
\email{pramod@math.lsu.edu}
\thanks{The research of the first author was partially supported by NSF
grant~DMS-0500873.}
\author{Daniel S.~Sage}
\email{sage@math.lsu.edu}
\thanks{The research of the second author was partially supported by NSF
grant~DMS-0606300.}

\begin{document}

\begin{abstract}

  Let $X$ be a scheme of finite type over a Noetherian base scheme $S$
  admitting a dualizing complex, and let $U \subset X$ be an open set
  whose complement has codimension at least $2$.  We extend the
  Deligne-Bezrukavnikov theory of perverse coherent sheaves by showing
  that a coherent intermediate extension (or intersection cohomology)
  functor from perverse sheaves on $U$ to perverse sheaves on $X$ may
  be defined for a much broader class of perversities than has
  previously been known.  We also introduce a derived category version
  of the coherent intermediate extension functor.

  Under suitable hypotheses, we introduce a construction (called
  ``$S_2$-exten\-sion'') in terms of perverse coherent sheaves of algebras
  on $X$ that takes a finite morphism to $U$ and extends it in a
  canonical way to a finite morphism to $X$.  In particular, this
  construction gives a canonical ``$S_2$-ification'' of appropriate
  $X$.  The construction also has applications to the
  ``Macaulayfication'' problem, and it is particularly well-behaved
  when $X$ is Gorenstein.

  Our main goal, however, is to address a conjecture of Lusztig on the
  geometry of special pieces (certain subvarieties of the unipotent
  variety of a reductive algebraic group).  The conjecture asserts in
  part that each special piece is the quotient of some variety
  (previously unknown in the exceptional groups and in positive
  characteristic) by the action of a certain finite group.  We use
  $S_2$-extension to give a uniform construction of the desired
  variety.
\end{abstract}

\maketitle

\section{Introduction}
\label{sect:intro}
Let $X$ be a scheme of finite type over a Noetherian base scheme $S$
that admits a dualizing complex, and let $U \subset X$ be an open set
whose complement has codimension at least $2$.  Let $\tU$ be another
scheme, equipped with a finite morphism $\rho_1: \tU \to U$.  Consider
the problem of completing the following diagram in a canonical way:
\[
\xymatrix{
*+{\;\tU\;} \ar@{^{(}.>}@<-.5ex>[r]\ar[d]_{\rho_1} & \tX \ar@{.>}[d]^{\rho} \\ 
*+{\;U\;} \ar@{^{(}->}[r] & X
}
\]
In other words: ``Construct a canonical new scheme $\tX$ that contains $\tU$ as
an open subscheme, together with a finite morphism $\rho: \tX \to X$
that extends $\rho_1$.''  One may want to impose additional
conditions, such as requiring $\tX$ to obey a regularity condition or
requiring the fibers of $\rho$ to have a specified form.  Moreover,
the pair $(\tX,\rho)$ should satisfy an appropriate universal
property.  If a group
$G$ acts on $X$ with $U$ a $G$-subscheme and $\rho_1$ $G$-equivariant,
one would like the entire constructed diagram to be equivariant.  The
present paper is motivated by a specific instance of this problem,
arising in a conjecture of Lusztig on the geometry of special pieces
(see below for the definition) in reductive algebraic groups.

In this paper, we give a general construction (called
``$S_2$-extension'') of such a scheme $\tX$ and morphism $\rho: \tX
\to X$, using Deligne's theory of perverse coherent sheaves on $X$
(following Bezrukavnikov's exposition~\cite{bez:pc}), assuming that
the category of coherent sheaves on $X$ has enough locally free
objects.  (This includes, for example, quasiprojective schemes over
$S$.)  This theory parallels the theory of constructible perverse
sheaves with the major exception that the intermediate extension (or
intersection cohomology) functor is not always defined.  Indeed, in
\cite{bez:pc}, this functor is only defined in an equivariant setting
with strong restrictions on the group action.  In this paper, we first
show that the intermediate extension functor may be defined for a much
broader class of perversities.  In particular, we study two dual
perversities, called the ``$S_2$'' and ``Cohen-Macaulay''
perversities.

Next, we construct $\tX$ as the global Spec of a certain intersection
cohomology sheaf with respect to the $S_2$ perversity.  It will be
defined whenever $\rho_{1*}\cO_{\tU}$ satisfies certain homological
conditions that are weaker than satisfying Serre's condition $S_2$.
The scheme $\tX$ is locally $S_2$ outside of $\tU$; moreover, $\rho$
satisfies a universal property related to this condition, and in that
sense $\tX$ and $\rho$ are canonical.  In the particular case of
$\tU=U$ and $\rho_1$ the identity, we obtain a canonical
``$S_2$-ification'' of $U$.  This construction also has applications
to the ``Macaulayfication'' problem.  Indeed, we give necessary and
sufficient conditions for $X$ to have a universal finite
Macaulayfication (i.e., universal among appropriate finite morphisms
from Cohen-Macaulay schemes).

Third, we introduce a derived category version of the coherent
intermediate extension functor (from a suitable subcategory of the derived
category of coherent sheaves on $U$ to the derived category of
coherent sheaves on $X$), and we show that this functor induces an
equivalence of categories with its essential image.  One corollary of
this theorem is that when $X$ is Gorenstein, the coherent
intermediate extension functor restricted to Cohen-Macaulay sheaves on $U$
is independent of perversity.  Using
this, we show that with suitable assumptions on $\tU$ and $X$, the
scheme $\tX$ produced by $S_2$-extension is in fact Cohen--Macaulay or Gorenstein.

Our main goal, however, is to apply these results to the
aforementioned conjecture of Lusztig, which we now recall.  Let $G$ be
a reductive algebraic group over the algebraically closed field $k$,
and assume that the characteristic of $k$ is good for $G$.  Let $C_1$
be a special unipotent class of $G$ in the sense
of~\cite{lus:special}.  The \emph{special piece} containing $C_1$ is
defined by
\[
P = \bigcup C
\qquad
\text{\begin{tabular}{@{}c@{}}
where $C$ ranges over unipotent classes such that $C \subset
\overline{C_1}$\\ but $C \not \subset \overline{C'}$ for any special $C'
\subset \overline{C_1}$ with $C' \ne C_1$.\end{tabular}}
\]
Each special piece is a locally closed subvariety of $G$, and
according to a result of Spaltenstein~\cite{spalt:classes}, every
unipotent class in $G$ is contained in exactly one special piece.

In 1981, Lusztig conjectured that every special piece is rationally
smooth~\cite{lus:green}.  This conjecture can be verified in any
particular group by explicit calculation of Green functions, and
indeed, the conjecture was quickly verified for all the exceptional
groups following work of Shoji~\cite{shoji:green} and
Benyon--Spaltenstein~\cite{bs:green}.  In the classical groups,
however, new techniques were required.  In 1989, Kraft and Procesi,
relying on their own prior work on singularities of closures of
unipotent classes, proved a stronger statement: they showed that every
special piece in the classical groups is a quotient of a certain
smooth variety by a certain finite group $F$~\cite{kp:special}.  In
particular, this implies that special pieces are rationally smooth.

A natural question, then, is whether this stronger statement holds in
general.  The work of Kraft--Procesi makes extensive use of the
combinatorics available in the classical groups, so it is not at all
obvious how to generalize their construction to all groups.  However,
in 1997, Lusztig succeeded in characterizing the finite group $F$ in a
type-independent manner~\cite[Theorem~0.4]{lus:notes}; he identified
$F$ as a certain subgroup of $\bar A(C_1)$.  (For any unipotent class
$C$, $\bar A(C)$ denotes Lusztig's ``canonical quotient'' of the
component group $G^x/(G^x)^\circ$ of the $G$-stabilizer of a point $x
\in C$.)  In fact, $F$ is naturally a direct factor of $\bar A(C_1)$
(see~\cite[\S 3.1]{as}) and inherits from $\bar A(C_1)$ the structure
of a Coxeter group (see~\cite{aa}).  There is a one-to-one
correspondence between parabolic subgroups of $F$ and unipotent
classes in $P$.  Given a parabolic subgroup $H \subset F$, we denote
the corresponding class $C_H$.  (The trivial subgroup corresponds to
the special class, so this notation is consistent with the earlier
notation $C_1$.)

\begin{conj}[Lusztig]\label{conj:lusztig}
There is a smooth variety $\tP$ with an action of $F$ such that $P \simeq
\tP/F$ and such that $C_H$ is precisely the image of those points in $\tP$
whose $F$-stabilizer is conjugate to $H$. 
\end{conj}

Note that it suffices to examine special pieces for the simple root systems because the unipotent variety of a reductive group is the product of the unipotent varieties of its simple factors.

Before addressing this conjecture further, we remark that it is quite easy
to produce candidate varieties that ought to be the preimages of the
various $C_H$'s in $\tP$ by using the results of~\cite{as} (a paper to
which the present paper might be regarded as a sort of
sequel). Fix $x\in C_H$.
By~\cite[Theorem~2.1]{as}, we have $\bar A(C_H) \simeq N_{\bar
A(C_1)}(H)/H$ (where $N_J(K)$ denotes the normalizer of $K$ in $J$) and
hence a natural map $G^x \to N_{\bar A(C_1)}(H)/H$.  Let $G^x_F$ be the
kernel of the composed map $G^x \to N_{\bar A(C_1)}(H)/H \to N_F(H)/H$, and
let $(\tC_H)^\circ = G/G^x_F$.  Clearly, this is a connected variety with a
free action of $N_F(H)/H$, and the quotient by that action is $C_H$. 
Finally, let 
\begin{equation}\label{eqn:preim-strat}
\tC_H = (\tC_H)^\circ \times_{N_F(H)} F.
\end{equation}
An element $a \in F$ acts on this variety by $a \cdot (y,f) = (y,
fa^{-1})$.  The $F$-stabilizer of any point is conjugate to $H$, and
the natural surjective map $\rho_H: \tC_H \to C_H$ is the quotient of
$\tC_H$ by the action of $F$.

Before stating our main result, we observe that, since quotients of
smooth varieties are normal, inherent in Lusztig's conjecture is the
subconjecture:
\begin{conj}\label{conj:normal} Every special piece $P$ is normal.
\end{conj}

In characteristic zero, we show how this conjecture can be obtained
from known results on unipotent conjugacy classes in the classical
types, $G_2$, $F_4$, and $E_6$.  For $E_7$ and $E_8$, there is a
conjectural list of all non-normal unipotent conjugacy class closures
due to Broer, Panyushev, and Sommers~\cite{broer:decomp}.  Assuming this is
true, then there would remain $5$ special pieces ($1$ in $E_7$ and $4$ in
$E_8$) for which normality is not known.  In positive characteristic,
much less is known.

In this paper, we will actually construct a variety $\tP$ whose
algebraic quotient by $F$ is the normalization $\bar P$ of $P$.
However, we will also show that special pieces are unibranch, i.e., the
normalization map $\nu:\bar P\to P$ is a bijection and in fact a
homeomorphism.  This means that $P$ is the topological quotient of
$\tP$.  In particular, setting $\bar{C}_H=\nu^{-1}(C_H)$, we see that
$\bar{C}_H\simeq C_H$ and that $\bar P$ is again stratified by
the unipotent orbits corresponding to parabolic subgroups of $F$.

The main result of the paper is the following.

\begin{thm}\label{thm:main}
\begin{enumerate}
\item There is a canonical normal irreducible $G$-scheme $\tP$
together with a finite equivariant morphism $\rho: \tP \to P$ which extends
$\rho_1:\tC_1\to C_1$; the pair $(\tP,\rho)$ is universal with respect
to finite morphisms $f:Y\to P$ that are $S_2$ relative to $C_1$ and
whose restriction $f|_{f^{-1}(C_1)}$ factors through $\rho_1$.
\item The variety $\tP$ is rationally smooth.  Moreover, if $\operatorname{char} k=0$, then $\tP$ is Gorenstein.
\item The variety $\tP$ is endowed with a natural $F$-action commuting
  with the $G$-action.  The map $\rho$ is the topological quotient by
  this action while $\bar\rho:\tP\to\bar{P}$ is the algebraic
  quotient.
\item For each class $C_H \subset P$, the preimage
  $\rho^{-1}(C_H)=\bar\rho^{-1}(\bar C_H)$ is isomorphic to $\tilde
  C_H$ and contains exactly those closed points whose $F$-stabilizer
  is conjugate to $H$.
\end{enumerate}
\end{thm}

The first part of this theorem is simply an invocation of the
$S_2$-extension construction.  The proof of the Gorenstein property is
established by using a theorem of Hinich and Panyushev~\cite{hinich,
  panyushev} and the aforementioned results on the derived
intermediate extension functor.  We remark that the formalism of the
$S_2$-extension construction does not yield a concrete
description of the resulting scheme in general,
but in our setting, the results of~\cite{as} (as noted above) allow us
to find an explicit stratification~\eqref{eqn:preim-strat} for $\tP$.

Although we do not prove that $\tP$ is smooth, we show that if $\hat
P$ is a smooth variety containing a dense open set isomorphic to
$\tC_1$ and $\hat\rho:\hat P\to P$ is a finite morphism extending
$\rho_1$, then $\hat P$ is isomorphic to $\tP$.  Thus, if Lusztig's
conjecture is true, then our $\tP$ is the desired smooth variety.  In
particular, for the classical groups, the $\tP$ constructed here
coincides with the Kraft--Procesi variety of~\cite{kp:special}.

\section{Perverse Coherent Sheaves}
\label{sect:perv-sheaves}

The theory of perverse coherent sheaves, following Deligne and
Bezrukavnikov~\cite{bez:pc}, closely parallels the much better-known
theory of constructible perverse sheaves, but one striking difference
is that in the coherent setting, the intermediate extension functor does not
always exist.  Indeed, in  loc.cit., it was only constructed in
an equivariant setting with strong assumptions on the group action.

In this section, we review the Deligne--Bezrukavnikov theory, and we prove a generalization of~\cite[Theorem~2]{bez:pc} that allows us to use the intermediate extension functor in a much broader class of examples, including many nonequivariant cases.

We begin with the same setting and assumptions as~\cite{bez:pc}.  Let
$X$ be a scheme of finite type over a Noetherian base scheme $S$
admitting a dualizing complex, and let $G$ be an affine group scheme
acting on $X$ that is flat, of finite type, and Gorenstein over $S$.
(For example, the base scheme could be $S = \Spec k$ with $k$ a field.)
By~\cite[Corollary V.7.2]{hartshorne}, a scheme $X$ satisfying these
assumptions necessarily has finite Krull dimension.  Let $\Coh(X)$ be
the category of $G$-equivariant coherent sheaves on $X$, and let
$\cD(X)$ be the bounded derived category of $\Coh(X)$.  We further
assume that $\Coh(X)$ has enough locally free objects.  Let $X^\topl$ be the
topological space consisting of generic points of $G$-invariant
subschemes of $X$, with the subspace topology induced by the
underlying topological space of $X$.  We adopt the convention that for any (not necessarily irreducible) $G$-invariant locally closed subscheme $Y \subset X$, the \emph{codimension} of $Y$ is given by
\[
\dc Y = \min_{y \in Y^\topl} \dim \cO_{y,X}.
\]
For any point
$x \in X$, let $i_x: \{x\} \to X$ denote the inclusion map.  (This is
merely a topological map, not a morphism of schemes.)
For brevity, we
will write $\bar x$ for the closed subspace $\overline{\{x\}}$ of $X$.
By~\cite[Proposition~1]{bez:pc}, $X$ admits an equivariant dualizing
complex $\omega_X$.  By shifting if necessary, we may assume, as
in~\cite[\S 3]{bez:pc}, that for each point $x \in X$, $i^!_x\omega_X$
is concentrated in degree $\dc \bar x$.

Although we will always work in this equivariant setting, one can of
course obtain nonequivariant versions of our results simply by taking $G =
1$.  Occasionally, we will explicitly pass from an equivariant category to
a nonequivariant one, and make use of the fact that all the usual functors
on sheaves commute with this forgetful functor.

\begin{conv} 
Throughout this paper, unless otherwise specified, all geometric
objects will belong to the appropriate category for the equivariant
setting without further mention.  Thus, schemes will be $G$-schemes,
morphisms will be $G$-morphisms, and sheaves will be $G$-equivariant.
\end{conv}

\begin{defn}
A \emph{perversity} is a function $p: X^\topl \to \Z$ satisfying
\begin{equation}\label{eqn:perv-defn}
\begin{aligned}
p(y) &\ge p(x) &
\quad &\text{and}\\
\dc \bar y - p(y) &\ge \dc \bar x - p(x) &
\qquad &\text{whenever $\dc \bar y \ge \dc \bar x$.}
\end{aligned}
\end{equation}
(In particular, $p(x)$ depends only on $\dc \bar x$.)  For any perversity $p$, the function $\bar p: X^\topl \to \Z$ defined by $\bar p(x) = \dc \bar x - p(x)$ is also a perversity, called the \emph{dual perversity} to $p$. 
\end{defn}

A slightly more general theory could be obtained by imposing the inequalities~\eqref{eqn:perv-defn} only when $y \in \bar x$, as is done in~\cite{bez:pc} (see also Remark~\ref{rmk:gen-perv}).  For the purposes of this paper, however, there would be no practical benefit to defining perversities in this way, and various technical details would become rather more complicated, so we will confine ourselves to perversities as defined above.

Given a perversity $p$, we define two full subcategories of $\cD(X)$ as follows:
\begin{align*}
\p\cD(X)^{\le 0} &= \{ \cF \in \cD(X) \mid \text{for all $x \in X^\topl$,
$H^k(i^*_x\cF) = 0$ for all $k > p(x)$} \} \\ 
\p\cD(X)^{\ge 0} &= \{ \cF \in \cD(X) \mid \text{for all $x \in X^\topl$,
$H^k(i^!_x\cF) = 0$ for all $k < p(x)$} \} 
\end{align*}
By~\cite[Theorem~1]{bez:pc}, $(\p\cD(X)^{\le 0}, \p\cD(X)^{\ge 0})$ is a
$t$-structure on $\cD(X)$.  

\begin{defn}
The above $t$-structure is called the \emph{perverse $t$-structure}
(with respect to the perversity $p$) on $\cD(X)$.  Its heart, denoted
$\cM^p(X)$ or simply $\cM(X)$, is the category of ($G$-equivariant)
\emph{perverse coherent sheaves}  on $X$ with respect to $p$.  The truncation functors for
this $t$-structure will be denoted $\ctrc^p_{\le 0}: \cD(X) \to
\p\cD(X)^{\le 0}$ and $\ctrc^p_{\ge 0}: \cD(X) \to \p\cD(X)^{\ge 0}$.
\end{defn}

We denote the standard $t$-structure on $\cD(X)$ by $(\std\cD(X)^{\le
  0}, \std\cD(X)^{\ge 0})$, and the associated truncation functors by
$\ctrc^\st_{\le 0}$ and $\ctrc^\st_{\ge 0}$.  The perverse
$t$-structure associated to the constant perversity $p = 0$ coincides
with the standard $t$-structure.

Now, let $U$ be a locally closed $G$-invariant subscheme of $X$, and
let $Z = \overline U \smallsetminus U$.  Let $U^\topl$ and $Z^\topl$
be the corresponding subspaces of $X^\topl$.  Given a perverse
coherent sheaf on $U$, we wish to find a canonical way to associate to
it a perverse coherent sheaf on $\overline U$, analogous to the
intermediate extension operation on ordinary (constructible) perverse
sheaves.  This is not always possible, but~\cite[Theorem~2]{bez:pc}
gives one set of conditions under which it can be done.  In fact, the conditions
of that theorem can be weakened significantly, at the expense of
having intermediate extension defined only on some subcategory of $\cM(U)$ (see Remark~\ref{rmk:p-pm}).

The following proposition provides a general framework for defining
intermediate extension on a subcategory of $\cM(U)$.  Later, we will determine the largest possible subcategory to which the proposition can be applied. 

Define a partial order on perversities by pointwise comparison: we say that
$p \le q$ if $p(x) \le q(x)$ for all $x \in X^\topl$. 

\begin{prop}\label{prop:interm-ext}
Suppose $q$, $p$, and $r$ are
perversities with the following properties: $q \le p \le r$, $r(x) - q(x)
\le 2$ for all $x$, and 
\[
q(x) = p(x) - 1
\qquad\text{and}\qquad
r(x) = p(x) + 1
\qquad\text{for all $x \in Z^\topl$.}
\]
Define two full subcategories by
\begin{align*}
\cM^{q,r}(U) &= {}^q\cD(U)^{\le0} \cap {}^r\cD(U)^{\ge 0} \subset
\cM^p(U), \\
\cM^{q,r}(\overline U) &= {}^q\cD(\overline U)^{\le0} \cap
{}^r\cD(\overline U)^{\ge 0} \subset \cM^p(\overline U), 
\end{align*}
and let $j: U \hookrightarrow \overline U$ be the inclusion map.  Then
$j^*: \cM^{q,r}(\overline U) \to \cM^{q,r}(U)$ is an equivalence of
categories. 
\end{prop}

\begin{defn}
The inverse equivalence to that of Proposition~\ref{prop:interm-ext}, which is denoted $\cIC^p(\overline U, \cdot): \cM^{q,r}(U) \to
\cM^{q,r}(\overline U)$, or simply $\cIC(\overline U, \cdot): \cM^{q,r}(U) \to
\cM^{q,r}(\overline U)$, is called the \emph{intermediate extension functor}. 
\end{defn}

\begin{proof}
Our proof is essentially identical to that of~\cite[Theorem~2]{bez:pc}. 
Let $J_{!*}: \cD(\overline U) \to \cD(\overline U)$ be the functor
$\ctrc^q_{\le 0} \circ \ctrc^r_{\ge 0}$.  We claim that $J_{!*}$ actually
takes values in $\cM^{q,r}(\overline U)$.  Given $\cF \in
\cD(\overline U)$, let $\cF_1 = \ctrc^r_{\ge 0}\cF$.  Then we have a
distinguished triangle 
\[
(\ctrc^q_{\ge 1}\cF_1)[-1] \to J_{!*}(\cF) \to \cF_1 \to \ctrc^q_{\ge 1}\cF_1.
\]
Note that $(\ctrc^q_{\ge 1}\cF_1)[-1] \in {}^q\cD(\overline U)^{\ge 2}$.  Now,
the condition $r(x) - q(x) \le 2$ implies that ${}^q\cD(\overline U)^{\ge 2}
\subset {}^r\cD(\overline U)^{\ge 0}$.  Clearly, $\cF_1 \in {}^r\cD(\overline
U)^{\ge 0}$, so it follows that $J_{!*}\cF \in {}^r\cD(\overline
U)^{\ge 0}$.  Since it obviously takes values in ${}^q\cD(\overline
U)^{\le 0}$, $J_{!*}\cF \in \cM^{q,r}(\overline U)$.

Next, note that if $\cF \in \cD(\overline U)$ is such that $\cF|_U \in
\cM^{q,r}(U)$, then both $(\ctrc^r_{\ge 0}\cF)|_U$ and $(\ctrc^q_{\le
0}\cF)|_U$, and hence $(J_{!*}\cF)|_U$, are isomorphic to $\cF|_U$.  In
particular, we can see now that $j^*$ is essentially surjective.  Given $\cF
\in \cM^{q,r}(U)$, let $\tilde \cF$ be any object on $\cD(\overline U)$
such that $j^*\tilde \cF \simeq \cF$.  (Such an object exists by~\cite[Corollary~2]{bez:pc}.)  Then $\cF' = J_{!*}\tilde\cF$ is an
object of $\cM^{q,r}(\overline U)$ such that $j^*\cF' \simeq \cF$. 

Now, if $\phi: \cF \to \cG$ is a morphism in $\cM^{q,r}(U)$, then
by~\cite[Corollary~2]{bez:pc}, we can find objects $\cF'$ and $\cG'$ in
$\cD(\overline U)$ and a morphism $\phi': \cF' \to \cG'$ such that $j^*\cF'
\simeq \cF$, $j^*\cG' \simeq \cG$, and $j^*\phi' \simeq \phi$.  By applying
$J_{!*}$, we may assume that $\cF'$, $\cG'$, and $\phi'$ actually belong to
$\cM^{q,r}(\overline U)$.  This shows that $j^*$ is full. 

To show that $j^*$ is faithful, it suffices to show that if $\phi$ is
an isomorphism, then $\phi'$ must be as well.  Since $\phi'|_U$ is an
isomorphism, the kernel and cokernel of $\phi'$ must be supported on
$Z$.  But by~\cite[Lemma~6]{bez:pc}, the fact that $q(x) < p(x) <
r(x)$ for $x \in Z^\topl$ implies that $\cF'$ and $\cG'$ have no
subobjects or quotients supported on $Z$.  Thus, $\phi'$ is an
isomorphism.  Since $j^*$ is fully faithful and essentially
surjective, it is an equivalence of categories.
\end{proof}

\begin{rmk}\label{rmk:interm-ext-calc}
  It follows from the above proof that for any $\cF \in \cM^{q,r}(U)$,
  $\cIC(\overline U,\cF)$ is isomorphic to $\ctrc^q_{\le 0} \ctrc^r_{\ge 0} \tilde
  \cF$, where $\tilde \cF$ is any object of $\cD(\overline U)$ whose
  restriction to $U$ is isomorphic to $\cF$.

  The above proof could also have been carried out using the functor
  $J_{!*}' = \ctrc^r_{\ge 0} \circ \ctrc^q_{\le 0}$ instead of
  $J_{!*}$.  From that version of the proof, one sees that $\cIC(\overline U,\cF)$
  is also isomorphic to $\ctrc^r_{\ge 0} \ctrc^q_{\le 0} \tilde \cF$.
\end{rmk}

\begin{prop}\label{prop:perv-pm}
Let $p$ be a perversity, and let $z_0$ be a generic
point of an irreducible component of $Z$ of
minimal codimension.  Among all perversities $q$ which,
together with some $r$, satisfy the assumptions of
Proposition~\ref{prop:interm-ext}, there is a unique maximal one, denoted
$p^-$.  It is given by 
\begin{align}
p^-(x) &=
\begin{cases}
p(x) - 1 & \text{if $p(x) \ge p(z_0)$,} \\
p(x) & \text{if $p(x) < p(z_0)$.}
\end{cases}
\intertext{Similarly, there is a unique minimal perversity among all $r$ of
that proposition, denoted $p^+$, and given by}
p^+(x) &=
\begin{cases}
p(x) + 1 & \text{if $\dc \bar x - p(x) \ge \dc \bar z_0 - p(z_0)$,} \\
p(x) & \text{if $\dc \bar x - p(x) < \dc \bar z_0 - p(z_0)$.}
\end{cases}
\end{align}
\end{prop}

\begin{rmk}\label{rmk:p-pm}
Although our formulas for $p^-$ and $p^+$ appear to be different from those
of~\cite[Theorem~2]{bez:pc}, they do in fact coincide under the assumptions
of loc. cit.  Those assumptions are that $U$ is open
and dense in $X$ and that for any $x \in
U^\topl$ and any $z \in \bar x \cap Z^\topl$, we have
\[
p(z) > p (x)
\qquad\text{and}\qquad
\dc \bar z - p(z) > \dc \bar x - p(x).
\]
These inequalities cannot hold simultaneously unless $\dc \bar x \le
\dc \bar z - 2$.  In particular, this means $U^\topl$ cannot contain
any closed points of $U$, so one must necessarily be in an equivariant
setting.

Proposition~\ref{prop:perv-pm}, on the other hand, applies with no a
priori restrictions on $X$ or $U$.  This really does allow us to use
the intermediate extension functor in nonequivariant settings, but in
practice, it is still necessary to require that $\dc U \le \dc Z - 2$;
indeed, if this condition fails, then $\cM^{p^-,p^+}(U)$ will be
reduced to the zero object.  To see this, note that $p^-(x)=p^+(x)$
implies that $\dc\bar x\le \dc\bar z_0 - 2$, so if $\dc U > \dc Z -
2$, then we have $p^-(x) < p^+(x)$ for all points $x \in U$.  It
follows that ${}^{p^-}\cD(U)^{\le 0} \subset {}^{p^+}\cD(U)^{\le -1}$,
so any object in $\cM^{p^-,p^+}(U)$ will belong to
${}^{p^+}\cD(U)^{\le -1} \cap {}^{p^+}\cD(U)^{\ge 0}$.  The latter
category contains only the zero object.
\end{rmk} 

\begin{proof}
Let us first show that $p^-$ is a perversity.  Suppose $\dc \bar x \ge
\dc \bar y$, so $p(x) \ge p(y)$.  If $p(x) \ge p(y) \ge
p(z_0)$ or $p(z_0) > p(x) \ge p(y)$, then the
conditions~\eqref{eqn:perv-defn} obviously hold because they hold for $p$. 
Now suppose $p(x) \ge p(z_0) > p(y)$.  The strictness of the second
inequality implies that $\dc \bar x > \dc \bar y$.  In this situation, we
clearly have $p^-(x) = p(x) - 1 \ge p(y) = p^-(y)$ and 
\[
\dc \bar x - p^-(x) = \dc \bar x - p(x) + 1 > \dc \bar y - p(y) = -
\dc \bar y - p^-(y). 
\]
Thus, $p^-$ is a perversity.

Let $q$ and $r$ be perversities satisfying the assumptions of
Proposition~\ref{prop:interm-ext}.  The
requirement that $q(x) = p(x) - 1$ for all $x \in Z^\topl$
implies that $q(x) = p(x) - 1 = p^-(x)$ for all $x$ with $\dc \bar x \ge
\dc \bar z_0$.  For all such points, of course, we have $p(x) \ge p(z_0)$.
 Now suppose $x$ is such that $\dc \bar x < \dc \bar z_0$, so that $p(z_0)\ge p(x)$.  If $p(z_0) > p(x)$, it is trivial that
$q(x) \le p^-(x)$, while if $p(z_0) = p(x)$, then  $q(x) \le q(z_0) =
p(z_0) - 1 = p(x) - 1 =
p^-(x)$.   Thus, $q(x) \le p^-(x)$ for all $x \in X^\topl$, so $q \le p^-$, and $p^-$ has the desired maximality
property. 

The proofs of the corresponding statements for $p^+$ are similar.
\end{proof}

\begin{rmk}\label{rmk:gen-perv}
If we were to change the definition of ``perversity'' by imposing the inequalities~\eqref{eqn:perv-defn} only when $y \in \bar x$, then
this result could be improved,  i.e., $p^-$ could be replaced be a larger
perversity and $p^+$ by a smaller one, resulting in a larger domain
category for $\cIC(\overline U, \cdot)$.  Let us call a sequence of points $x_1, y_1, x_2,
\ldots, y_k, x_{k+1}$ in $X^\topl$ a \emph{lower chain} (resp.~\emph{upper
chain}) if the following conditions hold: 
\begin{enumerate}
\item $x_i, x_{i+1} \in \bar y_i$ for all $i$, and $x_{k+1} \in \bar y_k
\cap Z^\topl$, and
\item $p(x_{i+1}) = p(y_i)$ and $p(x_i) = p(y_i) - \dc \bar y_i + \dc
\bar x_i$ (resp.~$p(x_{i+1}) = p(y_i) - \dc \bar y_i + \dc \bar x_{i+1}$
and $p(x_i) = p(y_i)$) for all $i$. 
\end{enumerate}
Let $S$ (resp.~$T$) be the set of all points of $X^\topl$ occurring in some
lower (resp.~upper) chain, and define 
\[
p^\ominus(x) =
\begin{cases}
p(x) - 1 & \text{if $x \in S$,} \\
p(x) & \text{otherwise,}
\end{cases}
\qquad\text{and}\qquad
p^\oplus(x) =
\begin{cases}
p(x) + 1 & \text{if $x \in T$,} \\
p(x) & \text{otherwise.}
\end{cases}
\]
It is not difficult to prove an analogue of Proposition~\ref{prop:perv-pm}
using these formulas instead of $p^-$ and $p^+$. 
\end{rmk}

\section{Notation and Preliminaries}
\label{sect:notation}

In this section, we introduce some useful notation and terminology,
and we prove a number of lemmas on perverse coherent sheaves.  To
simplify the discussion, we henceforth assume that $U$ is actually an
open dense subscheme of $X$ and that $Z$ has codimension at least $2$.
Let $j: U \hookrightarrow X$ be the inclusion map.  For
the most part, we will consider only ``standard'' perversities,
defined as follows.

\begin{defn}
A perversity $p$ is said to be \emph{standard} if
\begin{equation}\label{eqn:norm-perv}
p(x) = p^-(x) = p^+(x) = 0
\qquad\text{if $\dc \bar x = 0$.}
\end{equation}
\end{defn}
Note that if $p$ is standard, so is its dual $\bar p$.

\begin{rmk}
The assumption that $\dc Z \ge 2$ is
essential: if this condition fails, there are no standard
perversities.
\end{rmk}

Given a perversity $p$, when there is no risk of ambiguity, we write
\[
\cD(X)^{-,\le0} = {}^{p^-}\cD(X)^{\le 0}
\qquad\text{and}\qquad
\cD(X)^{+,\ge0} = {}^{p^+}\cD(X)^{\ge 0},
\]
or even simply $\cD^{-,\le 0}$ and $\cD^{+,\ge 0}$.  Next, let
\[
\cM^{p,\pm}(U) = \cD(U)^{-,\le 0} \cap \cD(U)^{+,\ge 0}
\qquad\text{and}\qquad
\cM^{p,\pm}(X) = \cD(X)^{-,\le 0} \cap \cD(X)^{+,\ge 0}.
\]
Then we have an intermediate extension functor $\cIC(X,\cdot): \cM^{p,\pm}(U) \to
\cM^{p,\pm}(X)$.  These categories will usually be denoted simply $\cM^\pm(U)$ and $\cM^\pm(X)$, respectively.

Let $\D$ denote the coherent (Serre-Grothendieck) duality functor $\cRHom(\cdot,\omega_X)$.  By~\cite[Lemma~5]{bez:pc}, $\D$ takes $\cM^p(X)$ to $\cM^{\bar p}(X)$ and $\cM^{p,\pm}(\overline U)$ to $\cM^{\bar p, \pm}(\overline U)$.

Two specific standard perversities will be particularly useful in the
sequel:
\begin{align*}
s(x) &=
\begin{cases}
0 \\
1 \end{cases}
&
c(x) &= 
\begin{cases}
\dc \bar x & \qquad\text{if $\dc \bar x < \dc Z$,} \\
\dc \bar x - 1 & \qquad\text{if $\dc \bar x \ge \dc Z$.}
\end{cases}
\\
\intertext{We call $s$ the ``$S_2$ perversity'' and $c$ the
``Cohen--Macaulay
perversity'' for reasons that are made clear in Lemma~\ref{lem:Ox-IC}. 
These two perversities are dual to one another, and they are extremal
among all standard perversities (see Lemma~\ref{lem:s2-cm-extreme}).  For convenience, we also
record the corresponding ``$-$'' and ``$+$'' perversities:} 
s^-(x) &= 0
&
c^+(x) &= \dc \bar x \\
s^+(x) &=
\begin{cases}
0 \\
1 \\
2 
\end{cases}
&
c^-(x) &=
\begin{cases}
\dc \bar x & \text{if $\dc \bar x < \dc Z + 1$,} \\
\dc \bar x - 1 & \text{if $\dc \bar x = \dc Z + 1$,} \\
\dc \bar x - 2 & \text{if $\dc \bar x \ge \dc Z$.}
\end{cases}
\end{align*}

It is clear that $s^-$ and $c^+$ are the smallest and largest possible
perversities, respectively, that take the value $0$ on generic points of
$X$.  Since the ``$-$'' and ``$+$'' operations respect the partial
order on perversities, we have the following result.

\begin{lem}\label{lem:s2-cm-extreme}
Every standard perversity $p$ satisfies $s \le p \le c$.\qed
\end{lem}

We will use the following observation repeatedly.

\begin{lem}\label{lem:icchar}
Let $\cF$ be a coherent sheaf on $U$.  The complex $\cIC(X, \cF)$ is defined if and
only if  $\cF \in \cD(U)^{+,\ge 0}$, or equivalently, if $\depth_{\cO_x} \cF_x \ge p^+(x)$ for all $x \in U^\topl$.

If $\cIC(X,\cF)$ is defined, then given a coherent extension $\cG$ of $\cF$ to $X$, we have $\cG \simeq \cIC(X,\cF)$ if and only if $\cG \in\cD(X)^{+,\ge 0}$, or equivalently, if $\depth_{\cO_x} \cG_x > p(x)$ for all $x \in Z^\topl$.
\end{lem}
\begin{proof}
The exactness of $i^*_x$ implies that $H^k(i^*_{x}\cF)=0$ unless
$k=0$, and since we are assuming that $p$ is a standard perversity, this
cohomology vanishes for $k>p^-(x)\ge 0$.  Thus, $\cF\in\cD(U)^{-,\le
0}$ automatically.  The depth-condition characterization of $\cD(U)^{+,\ge0}$ comes from the well-known fact that the lowest degree in which $H^k(i^!_x\cF)$ is nonzero is $\depth_{\cO_x} \cF_x$.  The same arguments apply to $\cG$ as well.
\end{proof}

Next, recall
(see~\cite[\S I.3.3]{eh:schemes}) that to any quasicoherent sheaf
of algebras $\cF$ on $X$, one can canonically associate a new scheme $Y$
and an affine morphism $f: Y \to X$ such that $f_*\cO_Y \simeq \cF$. 
Moreover, $f$ is finite if and only if $\cF$ is coherent. This procedure is
often called ``global Spec''; we will use the notation $Y = \bSpec \cF$.

Coherent sheaves of algebras and the global Spec operation play a major role in the sequel.  The following proposition relates these to the $\cIC$ functor.

\begin{prop}\label{prop:Spec-IC}
Let $\cF$ be a coherent sheaf of algebras on $X$.  Form $Y = \bSpec \cF$,
and let $f: Y \to X$ be the canonical map.  Then $\cIC(Y,\cO_{f^{-1}(U)})$ is defined if and only if $\cIC(X,\cF|_U)$ is.  If both are defined, then $\cIC(Y, \cO_{f^{-1}(U)})
\simeq \cO_Y$ if and only if $\cIC(X, \cF|_U) \simeq \cF$. 
\end{prop}
\begin{rmk}
  Here, the notation $\cIC(Y,\cO_{f^{-1}(U)})$ is to be understood in
  terms of the intermediate extension functor associated to the open
  inclusion $f^{-1}(U) \hookrightarrow Y$ and the perversity $p' = p
  \circ f: Y^\topl \to Z$.  The fact that $f$ is a finite morphism
  implies that the complement of $f^{-1}(U)$ has the same codimension
  as $Z$.  In addition, $\dc f^{-1}(\bar x) = \dc \bar x$ for any $x
  \in X^\topl$, so $p'$ does indeed satisfy the
  inequalities~\eqref{eqn:perv-defn}.

Moreover, since $Y$ is finite over $X$, it satisfies our basic
hypotheses for defining perverse coherent sheaves---it is of finite type over
$S$, and $\Coh(Y)$ has enough locally free objects.  (To see the
latter, note that if $\cG$ is a coherent $\cO_Y$-module, then there is
a locally free $\cO_X$-module $\cF$ which surjects to $f_*\cG$.  This
gives a surjective morphism of $\cO_Y$-modules $f^*\cF\to f^*f_*\cG$.
Composing with the surjection $f^*f_*\cG\to \cG$ exhibits $\cG$ as a
quotient of the locally free $\cO_Y$-module
$f^*\cF$.) 
\end{rmk}
\begin{proof}
By Lemma~\ref{lem:icchar}, it suffices to show $\cO_{f^{-1}(U)} \in \cD(f^{-1}(U))^{+,\ge0}$ if and only if $\cF|_U \in \cD(U)^{+,\ge0}$, and then 
that $\cO_Y \in \cD(Y)^{+,\ge0}$ if and only if
$\cF \in \cD(X)^{+,\ge0}$.  We prove both assertions simultaneously.

Let $x \in X^\topl$, and let $Y_x
= f^{-1}(x)$.  The latter is a finite set of points, and $(f|_{Y_x})_*$ is
clearly an exact functor that kills no nonzero sheaf.  Now,
we have
\[
R(f|_{Y_x})_* i^!_{Y_x} \cO_Y \simeq i^!_x f_* \cO_Y\simeq i^!_x \cF,
\]
so the lowest degree in which $H^k(i^!_x\cF)$ is nonzero is the same as the
lowest degree in which $H^k(i^!_{Y_x}\cO_Y)$ is nonzero.  Let $i_{y,Y_x}$
be the inclusion of a point $y$ into $Y_x$.  Then $i^!_{y,Y_x} =
i^*_{y,Y_x}$ is also an exact functor; it kills no nonzero sheaf
whose support contains $y$.  Since $i^!_y = i^!_{y,Y_x} \circ i^!_{Y_x}$,
we conclude that the lowest degree in which some $H^k(i^!_y \cO_Y) =
H^k(i^!_{y,Y_x}i^!_{Y_x} \cO_Y)$ is nonzero is the same as the lowest
degree in which $H^k(i^!_x \cF)$ is nonzero.  In particular, considering
the degree $k = p^+(x)$, we see that $\cO_Y \in \cD(Y)^{+,\ge0}$
if and only if $\cF \in \cD(X)^{+,\ge0}$, and likewise for $\cO_{f^{-1}(U)}$ and $\cF|_U$.
\end{proof}

Note that the proof in fact shows that if $f$ is a finite morphism,
then $f_*$ is $t$-exact.

In the remainder of the section, we prove a handful of results specific to the $S_2$ and Cohen--Macaulay perversities.

\begin{prop}\label{prop:S2-j*}
For any coherent sheaf $\cE$ on $U$ such that $\cIC^s(X, \cE)$ is
defined, there is a canonical isomorphism $\cIC^s(X, \cE) \simeq
j_*\cE$.  In particular, $j_*\cE$ is coherent.
\end{prop}
\begin{proof}
  We begin by observing that $\cIC^s(X,\cE)$ is actually a sheaf (
    i.e., that it is concentrated in degree $0$).  Indeed,
  $\cIC^s(X,\cE)$ is perverse with respect to $s^-$.  This perversity
  is constant with value $0$, so the resulting $t$-structure is just
  the standard $t$-structure.

Next, we show that $j_*\cE$ is coherent.  Note that the smallest value
taken by $s^+$ on $Z^\topl$ is $2$.  Now, $\cE$ is, by
assumption, a perverse coherent sheaf on $U$ with respect to $s^+$.
According to~\cite[Corollary~3]{bez:pc}, the complex $\ctrc\std_{\le
0} (Rj_*\cE)$ has coherent cohomology.  But
that object is simply $j_*\cE$, the nonderived push-forward of $\cE$.

Since $j_*\cE$ is concentrated in degree $0$, it obviously
lies in $\cD^{-,\le 0}(X)$, so by Remark~\ref{rmk:interm-ext-calc},
$\cIC^s(X,\cE)$ can be calculated as $\ctrc^+_{\ge0} (j_*\cE)$.  Thus, the truncation functor gives us a canonical morphism $j_*\cE \to \cIC^s(X,\cE)$.
On the other hand, we have the usual adjunction morphism
$\cIC^s(X,\cE) \to j_*j^{*}\cIC^s(X,\cE) \simeq j_*\cE$.  Both these
morphisms have the property that their restrictions to $U$ are simply
the identity morphism of $\cE$.  The compositions
\[
\cIC^s(X,\cE) \to j_*\cE \to \cIC^s(X,\cE)
\qquad\text{and}\qquad
j_*\cE \to \cIC^s(X,\cE) \to j_*\cE
\]
are then both identity morphisms of the appropriate objects, because
their restrictions to $U$ are the identity morphism of $\cE$, and the
functors $\cIC^s(X,\cdot)$ and $j_*$ are both fully faithful.  Thus,
$\cIC^s(X,\cE) \simeq j_*\cE$.
\end{proof}

We remark that the notation ``$\cIC^s(X,\cE)$'' is still useful, in
spite of the above proposition, because $\cIC^s(X,\cE)$ is not always
defined, whereas $j_*\cE$ is.  Most of the statements in
Section~\ref{sect:S2-ext} become false if we drop the assumption that
$\cIC(X,\cE)$ be defined and replace that object by $j_*\cE$, which is
not coherent in general.  Moreover, most of the proofs rely, at
least implicitly, on the fact that the $S_2$-perversity gives rise to
a nontrivial $t$-structure on $\cD(X)$.

\begin{defn}\label{defn:S2}
A scheme $X$ is said to be \emph{locally $S_2$} at $x \in X$ if $\depth
\cO_x \ge \min \{2, \dim \cO_x\}$.  $X$ is \emph{$S_2$} if it is locally
$S_2$ at every point.
\end{defn}

\begin{lem}\label{lem:Ox-IC}
$\cIC^s(X,\cO_U)$ is defined if and only if $U$ is locally $S_2$ at
all points $x \in U^\topl$ such that $\dc \bar x \ge \dc Z$.  In that
case, the following conditions are equivalent:
\begin{enumerate}
\item $\cIC^s(X,\cO_U) \simeq \cO_X$.\label{it:Ox-IC1}
\item $X$ is locally $S_2$ at all points of $Z^\topl$.\label{it:Ox-IC2}
\end{enumerate}
Similarly, $\cIC^c(X,\cO_U)$ is defined if and only if $U$ is
locally Cohen--Macaulay at all points of $U^\topl$.  In that case, the
following conditions are equivalent:
\begin{enumerate}
\item $\cIC^c(X,\cO_U) \simeq \cO_X$.
\item $X$ is locally Cohen--Macaulay at all points of $Z^\topl$.
\end{enumerate}
\end{lem}
\begin{proof}
The proofs of the two parts of this lemma are essentially identical; we
will treat only the $S_2$ case.  By Lemma~\ref{lem:icchar}, $\cIC^s(X,\cO_U)$ is defined if and only if $\depth \cO_x \ge s^+(x)$ for all $x \in U^\topl$,  i.e., if 
\[
\depth \cO_x \ge
\begin{cases}
0 & \text{if $\dc \bar x < \dc Z - 1$,} \\
1 & \text{if $\dc \bar x = \dc Z - 1$,} \\
2 & \text{if $\dc \bar x \ge \dc Z$.}
\end{cases}
\]
The first two cases above hold trivially.  (In the case $\dc \bar x = \dc
Z - 1$, we have $\dim \cO_x \ge 1$, and any local ring of positive
dimension has positive depth.)  Since $Z$ has codimension at least $2$,
the last case holds only if $\cO_x$ is $S_2$.  Thus, $\cIC^s(X,\cO_U)$ is
defined if and only if $U$ is locally $S_2$ at points $x \in X^\topl$ with
$\dc \bar x \ge \dc Z$.  The same argument applied to $\cO_X$ shows
the equivalence of 
conditions~\eqref{it:Ox-IC1} and~\eqref{it:Ox-IC2} above. 
\end{proof}

A similar proof gives the following result:

\begin{lem}\label{lem:F-IC}
  $\cIC^s(X,\cF)$ is defined if and only if $\depth \cF_x \ge \min
  \{2, \dim \cF_x\}$ at all points $x \in U^\topl$ such that $\dc
  \bar x \ge \dc Z$.
\end{lem}

Finally, for the last two lemmas of this section, we assume that $G$
is a linear algebraic group over $S = \Spec k$ for some algebraically
closed field $k$, and that $X$ is a variety 
over $k$.  In this case, we can extract a bit more geometric information from the
preceding results.  Recall that in this setting, the notion of
``orbit'' is well-behaved: $X$ is a union of orbits, each of which is
a smooth locally closed subvariety, isomorphic to a homogeneous space
for $G$.  The following lemma deals with local cohomology on an orbit.

\begin{lem}\label{lem:orbit-depth}
Let $\cF \in \cD(X)$, and let $C$ be a $G$-orbit in $X$.  Suppose that there is some $p \in \Z$ such that for any generic point $x$ of $C$, we have $H^k(i^!_x\cF) = 0$ for all $k < p$.  Then, for any
$y \in C$, we have $H^k(i^!_y\cF) = 0$ for all $k < p + \depth \cO_{y,C}$,
where $\cO_{y,C}$ is the local ring at $y$ of the reduced induced scheme
structure on $C$.
\end{lem}
\begin{proof}
We begin by noting that any equivariant coherent sheaf on $C$ is locally free.  Indeed,
given a coherent sheaf $\cE$ on $C$, consider the function $\phi: C \to \Z$
defined by $\phi(y) = \dim_{k(y)} k(y) \otimes_{\cO_{y,C}} \cE_y$, where
$k(y)$ is the residue field of the local ring $\cO_{y,C}$.  This function
is constant on closed points of $C$ (because they form a single $G$-orbit),
and hence, by the semicontinuity theorem, on all of $C$.  By, for
instance,~\cite[Ex.~II.5.8]{harts:ag}, since $\phi$ is constant and $C$ is reduced, $\cE$ is locally free.

Now, let $i_{y,C} : \{y\} \hookrightarrow C$ and $j_C : C \hookrightarrow
X$ be the inclusion maps.  (The former is merely a topological map; the
latter is a morphism of schemes.)  Recall that $i^!_{y,C} (j^!_C\cF) \simeq
\cRHom(\cO_{y,C}, i^!_y\cF)$, where $j^!_C$ is the right adjoint
to $Rj_{C*}$ in the setting of Grothendieck duality for coherent
sheaves (as constructed in, say,~\cite{hartshorne}), but $i^!_{y,C}$ and
$i^!_y$ are the Verdier-duality right adjoints to $(i_{y,C})_!$ and
$(i_y)_!$, respectively.

By the argument given in~\cite[Lemma~2(b)]{bez:pc}, the vanishing assumptions on $H^k(i^!_x\cF)$ for $x$ a generic point of $C$ imply that $H^k(j^!_C\cF)$
vanishes for all $k < p$; furthermore, the lowest nonzero cohomologies of
$i^!_{y,C}j^!_C\cF$ and of $i^!_y\cF$ occur in the same degree.  Now,
$j^!_C\cF$ is a bounded complex of locally free sheaves on $C$, so there
is some open subscheme $C_0 \subset C$ containing $y$ such that
$j^!_C\cF|_{C_0}$ is in fact a complex of free sheaves.  Recall, as in the proof of Lemma~\ref{lem:icchar}, that
$H^k(i^!_{y,C_0}\cO_{y,C_0})$ vanishes in degrees $k < \depth \cO_{y,C_0} =
\depth \cO_{y,C}$.  It follows that the cohomology of $i^!_{y,C}j^!_C\cF =
i^!_{y,C_0}(j^!_C\cF|_{C_0})$ vanishes in degrees $k < p + \depth
\cO_{y,C}$.
\end{proof}

We conclude with the following refinement of Lemma~\ref{lem:Ox-IC}.

\begin{lem}\label{lem:Ox-IC-orbit}
Assume that $G$ acts on $X$ with finitely many orbits.  If
$\cIC^s(X,\cO_U)$ is defined, then the following conditions are equivalent:
\begin{enumerate}
\item $\cIC^s(X,\cO_U) \simeq \cO_X$.
\item $X$ is locally $S_2$ at all points of $Z$.\label{it:Ox-IC3}
\end{enumerate}
Similarly, if $\cIC^c(X,\cO_U)$ is defined, the
following conditions are equivalent:
\begin{enumerate}
\item $\cIC^c(X,\cO_U) \simeq \cO_X$.
\item $X$ is locally Cohen--Macaulay at all points of $Z$.
\end{enumerate}
\end{lem}
\begin{proof}
As in the proof of Lemma~\ref{lem:Ox-IC}, we treat only the $S_2$
case.  Since $G$ acts with finitely many orbits, every closed
$G$-invariant subvariety contains an open orbit, so every point of
$X^\topl$ is a generic point of some $G$-orbit.  It suffices to show
that part~\eqref{it:Ox-IC2} of Lemma~\ref{lem:Ox-IC} is equivalent to
part~\eqref{it:Ox-IC3} of the present lemma.  That assertion follows
from Lemma \ref{lem:orbit-depth}: we see that
for any $x \in Z^\topl$ and any $y$ in the $G$-orbit $C$ containing $x$, we
have $\depth \cO_y \ge \depth \cO_x$, since $\depth \cO_{x,C} \ge 0$.
\end{proof}

\section{$S_2$-Extension}
\label{sect:S2-ext}

Our goal in this section is to use coherent intermediate extension with
respect to the $S_2$-perversity to construct new schemes and then to
use powerful general properties of the intermediate extension functor to
deduce various properties of those schemes.  Throughout this section,
all $\cIC$'s will be with respect to the $S_2$-perversity unless
otherwise specified.

The construction involves the global Spec operation (see
Proposition~\ref{prop:Spec-IC} and the comments preceding it) on
coherent sheaves of commutative algebras.  Henceforth, all sheaves of
algebras that we consider will be assumed to be coherent and
commutative.  We reemphasize the fact that we are working in the
equivariant setting, so that schemes are $G$-schemes, morphisms are
$G$-morphisms, and sheaves are $G$-equivariant.

\begin{prop}
  Let $\cE$ be a sheaf of $\cO_U$-algebras on $U$.  Then $\cIC(X,\cE)$
  can be made into a sheaf of $\cO_X$-algebras in a unique way that is
  compatible with the algebra structure on $\cE$.
\end{prop}
\begin{proof}
  This is immediate from Proposition~\ref{prop:S2-j*} and the fact
  that the algebra structure on $\cE$ determines a unique algebra
  structure on $j_*\cE$.
\end{proof}

\begin{defn}
Let $U \subset X$ be an open subscheme whose complement has codimension at
least $2$.  A morphism of schemes $f: Y \to X$ is said to be \emph{$S_2$
relative to $U$} if for all $x \in X^\topl$ such that $\dc \bar x \ge \dc Z$, we have $H^k(i^!_xf_*\cO_Y) = 0$ if $k < 2$.  
\end{defn}

\begin{rmk}\label{rmk:s2-gen}
  Note that if $f$ is finite and $S_2$ relative to $U$, then the image
  under $f$ of any generic point of an irreducible component of $Y$
  must lie in $U$.  Indeed, if $y$ is such a generic point and $x=f(y)$, then $H^0(i^!_y
  \cO_Y) \ne 0$, which implies that $H^0(i^!_xf_*\cO_Y) \ne 0$ by the
  argument given in the proof of Proposition~\ref{prop:Spec-IC}.  In
  particular, $f^{-1}(U)$ cannot be empty; in fact, it is open dense.
\end{rmk}

\begin{rmk}\label{rmk:s2-id}
  If $f$ is finite, the definition of ``$S_2$ relative to $U$'' is
  equivalent to requiring that $f_*\cO_Y \simeq \cIC(X,
  f_*\cO_{f^{-1}(U)})$, and hence, according to
  Proposition~\ref{prop:Spec-IC}, to requiring that $\cIC(Y,
  \cO_{f^{-1}(U)}) \simeq \cO_Y$.  (Note that the proposition applies
  since $f^{-1}(U)$ is open dense by the previous remark.)

  In particular, by Lemma~\ref{lem:Ox-IC}, $\id: X \to X$ is $S_2$
  relative to $U$ if and only if $\cIC(X,\cO_U)$ is defined and $X$ is
  locally $S_2$ outside $U$.  Moreover, if $f:Y\to X$ is a finite
  morphism with $Y$ $S_2$ and $f^{-1}(U)$ dense, then $f$ is $S_2$
  relative to $U$.
\end{rmk}

\begin{thm}\label{thm:S2-ext}
  Let $\rho_1: \tU \to U$ be a finite morphism such that $\cIC(X, \rho_{1*}\cO_\tU)$ is defined, and let $\tX$ denote
  the scheme $\bSpec \cIC(X, \rho_{1*}\cO_\tU)$.  The natural morphism
  $\rho: \tX \to X$ is universal with respect to finite morphisms
  $f:Y\to X$ which are $S_2$ relative to $U$ and whose restriction
  $f|_{f^{-1}(U)}$ factors through $\rho_1$.  In other words, if $f: Y
  \to X$ is any finite morphism that is $S_2$ relative to $U$ and such
  that $f|_{f^{-1}(U)}$ factors through $\rho_1$, then $f$ factors
  through $\rho$ in a unique way.

  In addition, $\rho$ is a finite morphism, $\rho^{-1}(U) \simeq \tU$, and
  $\rho|_\tU = \rho_1$.  Moreover, $\tU$ is a dense open subscheme of $\tX$, and $\id: \tX \to \tX$ is $S_2$ relative to $\tU$.
\end{thm}

Here is a diagram:
\[
\xymatrix@R-10pt{
\;f^{-1}(U)\; \ar@{^{(}->}[r]\ar[dr]^{g_1}\ar[ddr]_{f|_{f^{-1}(U)}}
& Y \ar@{.>}[dr]^{g}
\ar@/^.5pc/[ddr]|!{[d];[dr]}\hole_(.2){f} \\
& \;\tU\; \ar@{^{(}->}@<-.5ex>[r]_{\tj}\ar[d]^{\rho_1} & \tX\ar[d]^{\rho} \\
& \;U\; \ar@{^{(}->}[r]_{j} & X}
\]
As usual, the universal property enjoyed by $\tX$ and $\rho$ characterizes
them uniquely up to unique isomorphism.

Note that by Lemma~\ref{lem:F-IC}, the condition that $\cIC(X,
\rho_{1*}\cO_\tU)$ be defined is equivalent to requiring that $\depth
(\rho_{1*}\cO_\tU)_x \ge \min \{2, \dim (\rho_{1*}\cO_\tU)_x\}$ at all
points $x \in U^\topl$ such that $\dc \bar x \ge \dc Z$.  Moreover,
since $\cIC(X, \rho_{1*}\cO_\tU)\in \cM^{\pm}(X)$, this sheaf
satisfies the analogous depth conditions for all $x\in X^{\topl}$ with
$\dc \bar x \ge \dc Z$.  By Proposition~\ref{prop:S2-j*}, this is
equivalent to saying that $j \circ \rho_1: \tU \to X$ is $S_2$
relative to $U$.

\begin{defn}
The scheme $\tX$ constructed in Theorem~\ref{thm:S2-ext} is called the
\emph{$S_2$-extension} of $\rho_1: \tU \to U$.
\end{defn}

It should be noted that in general, an $S_2$-extension may be locally
$S_2$ only at the points in $X^\topl$.  It is in fact an $S_2$ scheme
when $G$ is trivial (so that $X^\topl = X$) or when
Lemma~\ref{lem:Ox-IC-orbit} can be invoked.  Moreover, if
$\cIC(X,\rho_{1*}\cO_\tU)$ is defined in the nonequivariant case,
i.e., the depth condition on $(\rho_{1*}\cO_\tU)_x$ described above
holds for $x\in U$ and not just $x\in U^\topl$, then both equivariant
and nonequivariant $S_2$-extensions of $\rho_1$ are defined.  Since
the nonequivariant universal mapping property is stronger than the
equivariant universal property, we obtain the following corollary.

\begin{cor} \label{cor:eqneq} If $\cIC(X,\rho_{1*}\cO_\tU)$ is defined in the
nonequivariant case, then the nonequivariant and equivariant
$S_2$-extensions of $\rho_1$ are canonically isomorphic.
\end{cor}

Before proving the theorem, we consider a few examples in which
$S_2$-extension has an elementary description.

\begin{exam}
If the map $j \circ \rho_1: \tU \to X$ is already finite (for example, if
$\tU$ is a single point), then $\tX = \tU$.
\end{exam}

\begin{exam}\label{exam:normal}
Note that the complement of $\tU$ in $\tX$ must have codimension at least $2$,
since $Z$ has codimension at least $2$ in $X$ and $\rho$ is finite.  Recall that according to Serre's criterion, a
scheme is normal if and only if it is $S_2$ and regular in codimension
$1$.  Thus, if $\tU$ is normal, the fact that $\tX$ is locally $S_2$
outside $\tU$ implies that $\tX$ is also normal.

In particular, suppose that $U$ is a normal subscheme of the integral
scheme $X$ and that $\rho_1: \tU \to U$ is an isomorphism.  Then
$\rho: \tX \to X$ is simply the usual normalization of $X$.  In view
of Proposition~\ref{prop:S2-j*}, we see that the normalization of $X$
has a remarkably simple description as $\bSpec(j_*\cO_U)$.
\end{exam}

\begin{exam}\label{exam:norm-cl}
  As a slight generalization of the previous example, let us now
  suppose only that $\tU$ is normal and that $X$ is integral.  An
  elementary construction of $\tX$ is given as follows.  Given an
  affine open subscheme $V = \Spec A$ of $X$, let $K$ be the fraction
  field of $\rho_1^{-1}(V)$, and let $B$ be the integral closure of
  the image of the natural map $A \to K$ induced by $\rho_1$.  Let
  $\tilde V = \Spec B$.  The various $\tilde V$'s obtained in this way
  as $V$ ranges over affine open subschemes of $X$ can be glued
  together to form a scheme $\tX'$.  This scheme enjoys a universal
  property similar to that of the normalization of a scheme.  By
  comparing with the universal property of $\tX$, it is easy to verify
  that $\tX$ and $\tX'$ are in fact canonically isomorphic.  We thus
  obtain an alternative elementary description of $\cIC(X,
  \rho_{1*}\cO_{\tU})$; it is the sheaf $V \mapsto B$.
\end{exam}

\begin{proof}[Proof of Theorem~\ref{thm:S2-ext}]
  We begin by establishing various properties of $\rho$ and $\tX$.
  Since $\cIC(X,\rho_{1*}\cO_\tU)$ is coherent, $\rho$ is finite.  From the definition of $\bSpec$, we know that
  $\rho^{-1}(U) \simeq \bSpec \cIC(X,\rho_{1*}\cO_\tU)|_U \simeq
  \bSpec \rho_{1*}\cO_\tU$.  Now, $\rho_1$ is finite, and therefore
  affine, so $\bSpec \rho_{1*}\cO_\tU$ is canonically isomorphic to
  $\tU$. Identifying these two schemes, we also see that $\rho|_\tU =
  \rho_1$.  Moreover, Proposition~\ref{prop:Spec-IC} tells us that
  $\cIC(\tX, \cO_\tU) \simeq \cO_\tX$, and then by Lemma~\ref{lem:Ox-IC} and Remark~\ref{rmk:s2-id}, we see that $\id: \tX \to \tX$ is $S_2$ relative to $\tU$.  

  As we have previously observed, $\rho$ finite implies that $\rho_*$
  is exact and $t$-exact.  Thus, $\rho_*\cO_\tX \simeq \cIC(X,
  \rho_{1*}\cO_\tU)$, and $\rho$ is $S_2$ relative to $U$.  We have
  already seen that $\rho|_\tU$ factors through $\rho_1$; indeed, with
  the obvious identifications, it equals $\rho_1$.  Finally,
  Remark~\ref{rmk:s2-gen} tells us that all generic points of
  irreducible components of $\tX$ lie in $\tU$, and hence that $\tU$
  is dense in $\tX$.

  It remains to show that $\tX$ and $\rho$ are universal with respect
  to these properties.  Let $f: Y \to X$ be a finite morphism that is $S_2$
  relative to $U$, and assume that $f|_{f^{-1}(U)}$ factors through
  $\rho_1$.  Let $V = f^{-1}(U)$, and let $g_0: V \to \tU$ be the
  morphism such that $f|_V = \rho_1 \circ g_0$.  Then $g_0$ gives rise
  to a morphism of sheaves $\rho_{1*}\cO_\tU \to f_*\cO_V$ on $U$ and
  therefore to a morphism of perverse coherent sheaves
\[
\rho_*\cO_\tX \simeq \cIC(X,\rho_{1*}\cO_\tU) \to \cIC(X, f_*\cO_V) \simeq
f_*\cO_Y.
\]
Applying the global Spec functor to this morphism of sheaves $\rho_*\cO_\tX \to
f_*\cO_Y$, we obtain a morphism of schemes $g: Y \to \tX$; this is the
desired morphism such that $f = \rho \circ g$.  The uniqueness of $g$
follows from the fact that there is a unique morphism $\cIC(X,
\rho_{1*}\cO_\tU) \to \cIC(X, f_*\cO_V)$ whose restriction to $U$ is the
morphism $\rho_{1*}\cO_\tU \to f_*\cO_V$ induced by $g_0$.
\end{proof}

As usual, any object characterized by a universal property comes with
a uniqueness theorem, but for $S_2$-extension, there is an even
stronger uniqueness property.

\begin{prop}\label{prop:unique}
Let $\hat X$ be a scheme containing $\tU$ as a dense open set,  and let
$\hat\rho:\hat X\to X$  be a finite morphism extending
$\rho_1:\tU\to U$.   If $\hat X$ is locally $S_2$ outside of $\tU$, then
$\hat X$ is isomorphic to $\tX$.
\end{prop}
\begin{proof}
  Since $\hat X$ is locally $S_2$ outside of $\tU$, its structure
  sheaf is an $\cIC$ sheaf: $\cO_{\hat X}\simeq\cIC(\hat
  X,\cO_{\tU})$.  The functor $\hat\rho_*$ is exact and $t$-exact, so
  $\hat\rho_*\cO_{\hat X}\simeq\cIC(X,\hat\rho_*\cO_{\tU})$.  But now
  $\hat X\simeq\bSpec \rho_*\cO_{\hat X}\simeq \tX$.
\end{proof}

\begin{rmk}
The developments of this section are closely related to the ideas in
Section~5.10 of EGA4, Part II~\cite{ega42}, which (translated into our
notation) deals with the $S_2$ condition for sheaves of the form
$j_*\cF$ and schemes of the form $\bSpec j_*\cO_U$.  The assumptions
in  loc. cit. are a bit different (e.g., the last part of our
Proposition~\ref{prop:S2-j*} must be imposed as a hypothesis), and the
specific setting of Theorem~\ref{thm:S2-ext} is not treated there.
Nevertheless, it is easy to imagine adapting the methods used there to
prove that $\rho: \tilde X \to X$ is $S_2$ relative to $U$.  However,
the universal property of $\rho$ and the uniqueness statement in
Proposition~\ref{prop:unique} are consequences of the fact that the
$\cIC$ functor is an equivalence of categories with its essential
image.  Proving those statements in the language
of~\cite[Section~5.10]{ega42} would likely amount to unwinding the
proof of Proposition~\ref{prop:interm-ext} and the construction of the
perverse coherent $t$-structure in~\cite{bez:pc}.  The conciseness and
clarity of the uniqueness arguments are perhaps the main benefit of
using perverse coherent sheaves here. 
\end{rmk}

\begin{rmk}
Our main goal in this paper is to apply the $S_2$-extension construction in
the setting of special pieces, where the hypotheses of
Example~\ref{exam:norm-cl} hold.  Indeed, $U$ will be normal, and $\rho_1:
\tU \to U$ will be a surjective \'etale morphism; see
Section~\ref{sect:tP}.  Since an elementary construction of
the $S_2$-extension is available in that setting (as was known to
C.~Procesi~\cite{procesi}), one could in principle forego developing the
machinery of the functor $\bSpec \cIC(X, \cdot)$. However, we will also
require the results of Section~\ref{sect:derived-IC}, and those seem to be
much easier to state and prove in the context of perverse coherent sheaves
than in a purely ring-theoretic setting.
\end{rmk}

We conclude this section with a remark on the ``Macaulayfication''
problem: given a scheme, find a Cohen--Macaulay scheme that is
birationally equivalent to it.  (For varieties over a field of
characteristic $0$, this problem is solved by Hironaka's Theorem.)
Kawasaki, extending early work of Faltings~\cite{faltings}, has shown
how to construct a Macaulayfication of any Noetherian scheme over a
ring (of arbitrary characteristic) with a dualizing
complex~\cite{kawasaki}, but the resulting scheme is not canonical.
(It does not have the obvious universal property.)  Indeed, just by
considering varieties that fail to be Cohen--Macaulay at a single
closed point, Brodmann has exhibited a family of examples which do not
have a universal Macaulayfication~\cite{brodmann}.  There may,
however, be a finite Macaulayfication that is universal among
appropriate finite morphisms from Cohen-Macaulay schemes.  The
following theorem addresses the problem in a way that is reminiscent
of Example~\ref{exam:normal}.

\begin{thm}[Macaulayfication]\label{thm:macaulay}
  Let $X$ be an scheme of finite type over a Noetherian base scheme
  $S$ admitting a dualizing complex, and suppose $\Coh(X)$ has enough
  locally free sheaves.  Let $U$ be an open Cohen--Macaulay subscheme
  whose complement has codimension at least $2$.  Then $X$ has a
  finite Macaulayfication if and only if $\cIC^c(X, \cO_U)$ is a sheaf
  (where $c$ is the Cohen--Macaulay perversity). In this case, the
  unique finite Macaulayfication is
  $\tX^{c}\overset{\mathrm{def}}{=}\bSpec \cIC^c(X,\cO_U)$ and
  coincides with the $S_2$-extension of $\mathrm{id}:U\to U$.  The scheme $\tX^{c}$ is
  universal with respect to finite morphisms to $X$ which are $S_2$
  relative to $U$.  In particular, any finite morphism $f:Y\to X$ with
  $Y$ Cohen-Macaulay and $f^{-1}(U)$ dense factors uniquely though
  $\tX^{c}$.
\end{thm}

This theorem is stated without a group action for convenience.  An
equivariant version akin to Theorem~\ref{thm:S2-ext} can be proved by
a similar argument.

\begin{proof}
  Let $f: Y \to X$ be a finite morphism such that $f^{-1}(U)$ is
  Cohen--Macaulay and dense in $Y$.  The last condition allows us to
  apply Lemma~\ref{lem:Ox-IC}: $Y$ is Cohen--Macaulay if and only if
  $\cIC^c(Y,\cO_{f^{-1}(U)}) \simeq \cO_Y$.  Next, by
  Proposition~\ref{prop:Spec-IC}, we see that $Y$ is Cohen--Macaulay
  if and only if we have $\cIC^c(X, f_*\cO_{f^{-1}(U)}) \simeq
  f_*\cO_Y$.  In particular, if $f$ is an isomorphism over $U$, then
  $Y$ is a Macaulayfication of $X$ if and only if $\cIC^c(X, \cO_U)$
  is a sheaf. However, since $c\ge s$, we know that whenever $\cIC^c(X, \cO_U)$
  is a sheaf, it must coincide with $\cIC^s(X, \cO_U)$.  We now see
  that $\tX^{c}$ is just the ``$S_2$-ification'' of $X$, and the universal
  property follows from Theorem~\ref{thm:S2-ext}.  Finally, a
  finite morphism $f:Y\to X$ with $Y$ Cohen--Macaulay and $f^{-1}(U)$
  dense is $S_2$ relative to $U$, so the universal property applies in
  this situation.
\end{proof}

Note that this construction does not coincide with Kawasaki's
Macaulayfication; the latter involves blow-ups and accordingly is
never finite.  Instead, this theorem generalizes a
result of Schenzel~\cite{schenzel} relating finite Macaulayfications
and ``$S_2$-ifications'' for a certain class of local rings.
(Schenzel's result is essentially Theorem~\ref{thm:macaulay} in the
special case $X = \Spec A$, where $A$ is a local domain that is a
quotient of a Gorenstein ring.)  However, the above scheme-theoretic
statement appears not to have been previously known.

Finally, we remark that Theorem~\ref{thm:S2-ext} can be generalized to
define a Cohen--Macaulay extension (or a $p$-extension for any
perversity $p$ with $s\le p\le c$) of appropriate $\rho_1:\tU\to U$.
Let $\rho_1$ be a finite morphism such that $\cIC^p(X,
\rho_{1*}\cO_\tU)$ is defined and a sheaf, and let $\tX^p$ denote the
scheme $\bSpec \cIC^p(X, \rho_{1*}\cO_\tU)$.  A similar argument to
that given in the proof of Theorem~\ref{thm:S2-ext} shows that the
natural morphism $\rho: \tX^p \to X$ is universal with respect to
finite morphisms $f:Y\to X$ which are ``$p$ relative to $U$'' (defined
in the obvious way) and whose restriction $f|_{f^{-1}(U)}$ factors
through $\rho_1$.  However, since $\cIC^p(X, \rho_{1*}\cO_\tU)$ and
$\cIC^s(X, \rho_{1*}\cO_\tU)$ coincide when the former is a sheaf, we
see that $\tX^p$ is just the $S_2$-extension, and so $\rho$ is in fact
universal with respect to finite morphisms $S_2$ relative to $U$ for
which $U$ has dense preimage.

\section{Middle Extension in the Derived Category}
\label{sect:derived-IC}

Recall that the intermediate extension functor $\cIC(X,\cdot): \cM^\pm(U) \to \cM^\pm(X)$ is an equivalence of categories.  (We make no assumptions about the perversity in this section.)  In particular, we have
\[
\Hom_{\cD(X)}(\cIC(X,\cE), \cIC(X,\cF)) \simeq \Hom_{\cM(U)}(\cE,\cF),
\]
for $\cE,\cF$ in $\cM^\pm(U)$.
The goal of this section is introduce a derived version of this
functor.  We will construct a functor of triangulated categories from a
suitable subcategory of $\cD(U)$ to $\cD(X)$ that ``extends''
$\cIC(X,\cdot)$, and then prove a generalization of
Proposition~\ref{prop:interm-ext} in this setting.

For most of this section, we make the
following assumption:
\begin{enumerate}
\item[(Q)] There is a class of projective objects $\cQ$ in $\cM^\pm(U)$ such
that\label{assum:proj}
{\makeatletter\@enumdepth=2\makeatother
\begin{enumerate}
\item every object of $\cM^\pm(U)$ is a quotient of some object in $\cQ$, and
\item for every object $\cA$ in $\cQ$, $\cIC(X,\cA)$ is a projective
sheaf on $X$.
\end{enumerate}}
\end{enumerate}

For example, condition~(Q) holds if $X$ is a quasiaffine scheme that is
locally $S_2$ outside $U$, and $p$ is the $S_2$ perversity.  In that case,
$\cIC(X,\cO_U) = \cO_X$.  Every object in $\cM^\pm(U)$ is in fact a
sheaf.  Moreover, since $U$ is quasiaffine, every
coherent sheaf on $U$ is a quotient of a free sheaf, so we can take $\cQ$
to be the class of free sheaves on $U$.

Let $\cM^\pm_0(U)$ be the abelian category $\cM^\pm(U) \cap
\Coh(U)$.  (In some cases, such as
with the $S_2$ perversity, it happens that $\cM^\pm_0(U) =
\cM^\pm(U)$.)  Let $\cM^\pm_0(X)$ be the subcategory of $\cM^\pm(X)$
consisting of objects $\cF$ such that $j^*\cF \in \cM^\pm_0(U)$.
Clearly, $\cIC(X,\cdot)$ and $j^*$ restrict to give equivalences of
categories between $\cM^\pm_0(U)$ and $\cM^\pm_0(X)$.

Now, $\cM^\pm_0(U)$ is a full subcategory of $\Coh(U)$ with enough
projective objects that are also projective in $\Coh(U)$
(namely, the objects of $\cQ$).  It follows that the bounded derived
category $\cD\cM^\pm_0(U)$ can be identified with a full triangulated
subcategory of $\cD(U)$.  For brevity, we henceforth write
$\cD^\pm_0(U)$ for $\cD\cM^\pm_0(U)$.

Let $\cD^\pm_0(X)$ be the full subcategory of $\cD(X)$ consisting of those
objects $\cA$ for which $\pH^n(\cA) \in \cM^\pm_0(X)$ for all $n$.  This is
a triangulated subcategory of $\cD(X)$ (because the subcategory
$\cM^\pm_0(X)$ of $\cM(X)$ is stable under extensions).  The perverse
$t$-structure on $\cD^\pm_0(X)$ (that is, the $t$-structure induced by the
perverse $t$-structure on $\cD(X)$) has heart $\cM^\pm_0(X)$.

Following~\cite[\S 3.1]{bbd}, given any $t$-structure on a full
triangulated subcategory of the derived category of an abelian category,
there is a \emph{realization functor} from the derived category of the
heart of the $t$-structure to the original derived category.  In our
situation, we obtain a functor $\real: \cD\cM^\pm_0(X) \to \cD(X)$, where
$\cD\cM^\pm_0(X)$ is the bounded derived category of the abelian category
$\cM^\pm_0(X)$.  We now briefly review its construction. This requires the
machinery of filtered derived categories; we refer the reader to~\cite[\S
3.1]{bbd} for complete definitions and details.  Let $\cD\cF(X)$ be the
bounded filtered derived category of coherent sheaves on $X$, and let
$\cD\cF_\bete(X)$ be the full subcategory of $\cD\cF(X)$ consisting of
objects whose filtration is stupid (``b\^ete'') with respect to the
perverse $t$-structure on $\cD^\pm_0(X)$. Forgetting the filtration gives
us a functor $\omega: \cD\cF_\bete(X) \to \cD(X)$; on the other hand,
by~\cite[Proposition~3.1.8]{bbd}, there is an equivalence of categories
$G: \cD\cF_\bete(X) \to \cC\cM^\pm_0(X)$, where $\cC\cM^\pm_0(X)$ is the
category of complexes of objects of $\cM^\pm_0(X)$.  We first define
$\overline{\real}: \cC\cM^\pm_0(X) \to \cD(X)$ by $\overline{\real} = \omega
\circ G^{-1}$. By~\cite[Proposition~3.1.10]{bbd}, $\overline{\real}$
factors through $\cD\cM^\pm_0(X)$ and thus gives rise to a functor $\real:
\cD\cM^\pm_0(X) \to \cD(X)$.  This functor is compatible with cohomology in the
following sense: for all $n$, there is an isomorphism of functors $H^n
\simeq \pH^n \circ \real: \cD\cM^\pm_0(X) \to \cM^\pm_0(X)$.
 
We define $\bIC(X,\cdot): \cD^\pm_0(U) \to \cD(X)$ by $\bIC(X,\cdot) = \real
\circ \cIC(X,\cdot)$.

\begin{lem}
The functor $\bIC(X,\cdot): \cD^\pm_0(U) \to \cD(X)$ takes values in
$\cD^\pm_0(X)$, and there are isomorphisms of functors $\pH^n(
\bIC(X,\cdot)) \simeq  \cIC(X,\pH^n(\cdot))$ for all $n$.
\end{lem}
\begin{proof}
Since $\cIC(X,\cdot): \cM^\pm_0(U) \to \cM(X)$ is exact, the functor on derived
categories $\cIC(X,\cdot): \cD^\pm_0(U) \to \cD\cM(X)$ respects cohomology: $\cIC(X,\pH^n(\cdot))$ (or, equivalently, $\cIC(X,\std H^n(\cdot))$) is isomorphic to $H^n(\cIC(X,\cdot))$.  Next, $H^n \simeq \pH^n \circ \real$, so
\[
\cIC(X,\pH^n(\cdot)) \simeq H^n(\cIC(X,\cdot)) \simeq \pH^n(\real(\cIC(X,\cdot))) \simeq \pH^n(\bIC(X,\cdot)).
\]
Since $\cIC(X,\cdot)$ takes values in $\cM^\pm_0(X)$, it is obvious that
$\bIC(X,\cdot)$ takes values in $\cD^\pm_0(X)$.
\end{proof}

\begin{lem}\label{lem:real-left-inv}
There is an isomorphism of functors $j^*\bIC(X,\cdot) \simeq \id:
\cD^\pm_0(U) \to \cD^\pm_0(U)$.
\end{lem}
\begin{proof}
Since $\cIC(X,\cdot): \cM^\pm_0(U) \to \cM^\pm_0(X)$ and $j^*: \cM_!*(U) \to
\cM^\pm_0(U)$ are inverse equivalences of categories, they
give rise to equivalences of the corresponding categories of complexes, as
well as of the the corresponding derived categories.  These equivalences
are such that the square in the center of the diagram below commutes.
\[
\xymatrix{
\cD\cF_\bete(X) \ar[d]_{j^*}^{\wr} \ar[r]_{G}^{\sim}
\ar@/^1.5pc/[rrr]^{\omega} &
\cC\cM^\pm_0(X) \ar@<-1ex>[d]_{j^*}^{\wr} \ar[r] &
\cD\cM^\pm_0(X) \ar@<-1ex>[d]_{j^*}^{\wr} \ar[r]^{\real} &
\cD(X) \ar@/^2.3pc/[dl]^{j^*} \\
\cD\cF_\bete(U) \ar[r]^{G}_{\sim} \ar@/_1pc/[rr]_{\omega}&
\cC\cM^\pm_0(U) \ar@<-1ex>[u]_{\cIC(X,\cdot)} \ar[r] &
\cD^\pm_0(U) \ar@<-1ex>[u]_{j^{*}} \ar@<-0.5ex>@/^1pc/[ur]_{\bIC(X,\cdot)}}
\]

In the setting of filtered derived categories, the restriction functor
$j^*: \cD\cF(X) \to \cD\cF(U)$ respects stupidity of the filtration
(because $j^*$ takes $\cM^\pm_0(X)$ to $\cM^\pm_0(U)$) and so gives
rise to a functor $j^*: \cD\cF_\bete(X) \to \cD\cF_\bete(U)$ that
makes the leftmost square in the diagram above commute.  Here
$\cD\cF_\bete(U)$ is defined with respect to the perverse
$t$-structure on $D^\pm_0(U)$ (which is simply a shift of the standard
$t$-structure).  It is clear that restriction commutes with forgetting
the filtration, so $j^* \circ \omega \simeq \omega \circ j^*$.
Together, these statements imply that $j^* \circ \real:
\cD\cM^\pm_0(X) \to \cD^\pm_0(U)$ is isomorphic to $j^*:
\cD\cM^\pm_0(X) \to \cD^\pm_0(U)$.  Composing with $\cIC(X,\cdot):
\cD^\pm_0(U) \to \cD\cM^\pm_0(X)$, we find that $j^* \circ \real
\circ {}\cIC(X,\cdot) \simeq j^* \circ \cIC(X,\cdot)$, or in other words, $j^* \circ \bIC(X,\cdot) \simeq \id$.
\end{proof}

\begin{defn}
An object $\cF$ of $\cM^\pm_0(U)$ is said to be \emph{short} if $\cIC(X,\cF)
\in \std \cD(X)^{\le 1}$.
\end{defn}

For example, if $p$ is the $S_2$ perversity, all objects in $\cM(U)$ are
short.  Indeed, they, as well as all other short objects we will actually
encounter, satisfy the stronger condition that their images under $\cIC(X,\cdot)
$ belong to $\std \cD(X)^{\le 0}$, but the weaker condition above
suffices for the statements we wish to prove.

\begin{prop}\label{prop:short-isom}
If $\cF \in \cM^\pm_0(U)$ is short, there are natural isomorphisms
\begin{align}
\Hom_{\cD(X)}(\cIC(X,\cE), \cIC(X,\cF)[n]) &\simeq \Hom_{\cD(U)}(\cE, \cF[n])
\label{eqn:hom-m} \\
\cRHom_{\cD(X)}(\cIC(X,\cE), \cIC(X,\cF)) &\simeq Rj_*\cRHom_{\cD(U)}(\cE,
\cF) \label{eqn:crhom-m}
\end{align}
for all $n \in \Z$ and all $\cE \in \cM^\pm_0(U)$.
\end{prop}

\begin{rmk}\label{rmk:bbd-3117}
According to~\cite[Remarque~3.1.17(ii)]{bbd}, the
isomorphism~\eqref{eqn:hom-m} always exists for $n \le 1$, without
assuming condition~(Q) or the shortness of $\cF$.
\end{rmk}

\begin{proof}
As we have just remarked, the natural morphism
\begin{equation}\label{eqn:real-full}
\Hom_{\cD(U)}(\cE, \cF[n]) \to
\Hom_{\cD(X)}(\cIC(X,\cE), \cIC(X,\cF)[n]).
\end{equation}
induced by $\cIC(X,\cdot)$ is an isomorphism for $n \le 1$. 
Lemma~\ref{lem:real-left-inv} implies that this morphism is always
injective (in other words, that $\cIC(X,\cdot)$ is faithful), so it simply
remains to show that it is surjective for $n > 1$.  We proceed by induction.

Given $f:
\bIC(X,\cE) \to \bIC(X,\cF)[n]$, choose a surjective map $g: \cA \to \cE$
with $\cA \in \cQ$.  By assumption, $\cIC(X,\cA)$ is a projective sheaf,
so $\Hom(\cIC(X,\cA),\cG) = 0$ for all $\cG \in \std\cD(X)^{\le -1}$. 
Since $\cF$ is short, we have $\cIC(X,\cF)[n] \in
\std\cD(X)^{\le -n+1}$, and since $n > 1$, we have
\[
\Hom(\cIC(X,\cA), \cIC(X,\cF)[n]) = 0.
\]
Now, let $\cH = \ker g$, and consider the exact sequence
\begin{multline*}
\cdots \to \Hom(\cIC(X,\cH)[1], \cIC(X,\cF)[n]) \to
\Hom(\cIC(X,\cE), \cIC(X,\cF)[n]) \\
\to \Hom(\cIC(X,\cA), \cIC(X,\cF)[n]) \to \cdots.
\end{multline*}
We see that $f$ must be the image of some morphism $f': \cIC(X,\cH)[1] \to
\cIC(X,\cF)[n]$, that is, $f = f' \circ d$, where $d: \cIC(X,\cE) \to \cIC(X,\cH)[1]$ comes
from the distinguished triangle associated to the short exact sequence
$0 \to \cH \to \cA \to \cE \to 0$.  Now, $f'[-1]: \cIC(X,\cH) \to \cIC(X,\cF)[n-1]$ is in the image of $\cIC(X,\cdot)$ by the
inductive hypothesis, and $d$ is in its image by
Remark~\ref{rmk:bbd-3117}, so $f$ is in its image as well.

It remains to prove the corresponding fact for $\cRHom$.  For the remainder
of the proof, we assume that we are working in the nonequivariant setting
(or that $G$ is trivial).  As remarked in~\cite{bez:pc} immediately
preceding Lemma~2, $\cRHom$ commutes with the forgetful functor from an
equivariant category to the nonequivariant one, so we lose nothing by
making this assumption. 

In the nonequivariant setting, all of the preceding arguments also apply to
any open set $V \subset X$.  In particular, since $\bIC(X,\cE)|_V \simeq
\bIC(V, \cE|_{V \cap U})$ for any object $\cE \in \cM^\pm_0(U)$, we have, for any $n \in \Z$, an
isomorphism 
\[
\Hom_{\cD(V)}(\cIC(X,\cE)|_V, \cIC(X,\cF)[n]|_V) \simeq \Hom_{\cD(V
\cap U)}(\cE|_{V \cap U},\cF[n]|_{V \cap U}).
\]

Now, for any two objects $\cA, \cB \in \cD(X)$, there is a canonical
morphism 
\[
j^*{\cRHom_{\cD(X)}(\cA,\cB)} \to \cRHom_{\cD(U)}(j^*\cA, j^*\cB).
\]
Let us take $\cA = \cIC(X,\cE)$ and $\cB = \cIC(X,\cF)$.  Of course, we
then have $j^*\cA \simeq \cE$ and $j^*\cB \simeq \cF$.  Now, by
adjointness, the above morphism gives rise to a canonical morphism 
\[
\phi: \cRHom_{\cD(X)}(\cIC(X,\cE),\cIC(X,\cF)) \to
Rj_*\cRHom_{\cD(U)}(\cE,\cF).
\]
To show that $\phi$ is in fact an isomorphism, it suffices to show that it
induces isomorphisms on all hypercohomology groups over all open sets.  For
any open set $V \subset X$, we have 
\begin{multline*}
H^n(R\Gamma(V, \cRHom_{\cD(X)}(\cIC(X,\cE), \cIC(X,\cF)))) \\
\simeq \Hom_{\cD(V)}(\cIC(X,\cE)|_V, \cIC(X,\cF)[n]|_V) 
\simeq \Hom_{\cD(V \cap U)}(\cE|_{V \cap U}, \cF[n]|_{V \cap U}) \\
\simeq H^n(R\Gamma(V, Rj_*\cRHom(\cE, \cF))).
\end{multline*}
Thus, $\phi$ is an isomorphism.
\end{proof}

\begin{thm}\label{thm:bj-equiv}
If all objects in $\cM^\pm_0(U)$ are short, then
$\bIC(X,\cdot): \cD^\pm_0(U) \to \cD^\pm_0(X)$ is an equivalence of categories, with
inverse given by $j^*$.  Moreover, for any two objects
$\cE, \cF \in \cD^\pm_0(U)$, there are natural isomorphisms
\begin{align}
\Hom_{\cD(X)}(\bIC(X,\cE), \bIC(X,\cF)) &\simeq \Hom_{\cD(U)}(\cE,\cF)
\label{eqn:hom-d}\\
\cRHom_{\cD(X)}(\bIC(X,\cE), \bIC(X,\cF)) &\simeq
Rj_*\cRHom_{\cD(U)}(\cE,\cF) \label{eqn:crhom-d}
\end{align}
\end{thm}
\begin{proof}
If all objects in $\cM^\pm_0(U)$ are short, the isomorphism~\eqref{eqn:hom-m} holds for
all objects $\cE, \cF \in \cM^\pm_0(U)$.  As observed
in the proof of~\cite[Proposition~3.1.16]{bbd}, the realization functor is
an equivalence of categories if and only if~\eqref{eqn:hom-m} holds for all
objects in $\cM^\pm_0(U)$.  Since $\cIC(X,\cdot): \cD^\pm_0(U) \to \cD\cM^\pm_0(X)$ is an
equivalence of categories, we see that $\bIC(X,\cdot) = \real \circ \cIC(X,\cdot)$ is
as well.  By Lemma~\ref{lem:real-left-inv}, its inverse must be $j^*$.

Once we know that $\bIC(X,\cdot)$ is an equivalence of
categories,~\eqref{eqn:hom-d} is immediate.  We deduce~\eqref{eqn:crhom-d}
from it by an argument identical to that given for~\eqref{eqn:crhom-m}
above.
\end{proof}

For applications of this result, we must consider certain categories whose objects are dual to coherent sheaves.  Given a perversity $p$, let
\[
\cM^{p,\pm}_*(U) = \D(\cM^{\bar p,\pm}_0(U)) \subset \cM^{p,\pm}(U).
\]

\begin{cor}\label{cor:dual-perv}
Suppose the dualizing complexes on $U$ and $X$ have the following properties: with respect to some perversity $p$, $\omega_U$ is a short object in $\cM^{p,\pm}_0(U)$, and $\omega_X \simeq \cIC^p(X,\omega_U)$.  Then, with respect to the dual perversity $\bar p$, we have $\cIC^{\bar p}(X,\cE) \simeq Rj_*\cE$ for all $\cE \in \cM^{\bar p,\pm}_*(U)$.
\end{cor}
\begin{proof}
Let $\cE$ be an object of $\cM^{\bar p,\pm}_*(U)$.  Then $\D\cE =
\cRHom_{\cD(U)}(\cE, \omega_U)$ is an object in $\cM^{p,\pm}_0(U)$.  It follows from~\cite[Lemma~5(a)]{bez:pc} that
\[
\D\cIC^p(X,\D\cE) = \cRHom(\cIC^p(X,\D\cE), \omega_X) \simeq \cIC^{\bar p}(X, \cE).
\]
But we also have
\begin{align*}
\cRHom_{\cD(X)}(\cIC^p(X,\D\cE), \omega_X) &\simeq
\cRHom_{\cD(X)}(\cIC^p(X,\D\cE), \cIC^p(X,\omega_U)) \\
&\simeq Rj_*\cRHom_{\cD(U)}(\D\cE, \omega_U).
\end{align*}
Since $\cRHom_{\cD(U)}(\D\cE, \omega_U) \simeq \cE$, we see that $\cIC^{\bar p}(X,\cE) \simeq Rj_*\cE$.
\end{proof}

\begin{cor}\label{cor:Gor-perv}
Suppose $X$ is a Gorenstein scheme.
\begin{enumerate}
\item If $p^+(x) = \dc \bar x$ for all $x \in U^\topl$, then
there is an isomorphism of functors $\cIC(X,\cdot) \simeq Rj_*$.\label{it:arboreal}
\item If $\cF$ is a Cohen--Macaulay sheaf on $U$, then $\cIC(X,\cF)
\simeq Rj_*\cF$ with respect to any perversity.\label{it:cm}
\end{enumerate}
\end{cor}
\begin{proof}
  On a Gorenstein scheme, we may take $\omega_X \simeq \cO_X$.  A Gorenstein
  scheme is, in particular, Cohen--Macaulay, so by
  Lemma~\ref{lem:icchar}, $\omega_X \simeq \cIC(X,\omega_U)$
  with respect to every perversity.  Corollary~\ref{cor:dual-perv}
  now tells us that on $\cM^\pm_*(U)$, $\cIC(X,\cdot) \simeq Rj_*$ for every
  perversity.

For part~\eqref{it:arboreal} of the corollary, we must simply show that $\cM^{p,\pm}_*(U) = \cM^{p,\pm}(U)$, or, equivalently, that $\cM^{\bar p,\pm}_0(U) = \cM^{\bar p,\pm}(U)$.  The assumption that $p^+(x) = \dc \bar x$ implies that $\bar p^-(x) = 0$ for all $x \in U^\topl$.  It follows that $\cM^{\bar p,\pm}(U) \subset \cM^{p^-}(U) = \Coh(U)$, as desired.

For part~\eqref{it:cm}, we first note that $\cIC(X,\cF)$ is defined with respect to any perversity by Lemma~\ref{lem:icchar}.  In particular, we see that $\cF \in \cM^{c,\pm}(U)$, so $\cD\cF \in \cM^{s,\pm}(U) = \cM^{s,\pm}_0(U)$.  Since $\cD\cF \in \Coh(U)$, $\cF$ is in $\cM^\pm_*(U)$ with respect to any perversity, and the result follows by Corollary~\ref{cor:dual-perv}.
\end{proof}

\begin{cor}\label{cor:S2-Gor}
Let $X$ be a Gorenstein scheme, $U \subset X$ an open subscheme, and
$\rho_1: \tU \to U$ a finite morphism.  If $\rho_1$ admits an $S_2$-extension, let $\tX$ be the scheme thus obtained.
\begin{enumerate}
\item If $\tU$ is Cohen--Macaulay, then $\tX$ is as well.\label{it:S2-cm}
\item If $\rho_{1*}\cO_{\tU}$ is isomorphic to its own Serre--Grothendieck dual, then $\tX$ is Gorenstein.\label{it:S2-Gor}
\end{enumerate}
\end{cor}
In particular, part of the content of this corollary is the assertion that if either $\tU$ is Cohen--Macaulay or $\rho_{1*}\cO_{\tU}$ is self-dual, then $\rho_1$ necessarily admits an $S_2$-extension.

\begin{proof}
For part~\eqref{it:S2-cm}, it follows from
Proposition~\ref{prop:Spec-IC}, Lemma~\ref{lem:Ox-IC}, and
Corollary~\ref{cor:Gor-perv} that the $S_2$-extension $\tX$ exists and
is locally Cohen--Macaulay at least at all points of $\tX^\topl$.  This
reasoning can be repeated in the nonequivariant category to obtain
a nonequivariant $S_2$-extension that is in fact Cohen--Macaulay.  The latter
variety must coincide with $\tX$ by Corollary~\ref{cor:eqneq}.

Henceforth, assume that
$\rho_{1*}\cO_{\tU} \simeq \D(\rho_{1*}\cO_{\tU})$.  Evidently, $\rho_{1*}\cO_{\tU} \in {}^{s^-}\cD(U)^{\le 0}$, and since the dual perversity to $s^-$ is $c^+$, we have $\rho_{1*}\cO_{\tU} \in {}^{c^+}\cD(U)^{\ge 0}$ as well.  It follows that the intermediate extension of $\rho_{1*}\cO_{\tU}$ is defined with respect to any perversity; furthermore, by Corollary~\ref{cor:Gor-perv}, it is independent of perversity.  Let $\cF = \cIC(X, \rho_{1*}\cO_{\tU})$.  By~\cite[Lemma~5]{bez:pc}, $\D\cF \simeq \cIC(X,\D(\rho_{1*}\cO_{\tU}))$, and hence $\cF \simeq \D\cF$.

Now, by Proposition~\ref{prop:Spec-IC} and Lemma~\ref{lem:Ox-IC},
we know that $\tX = \bSpec \cF$ is Cohen--Macaulay.  In particular,
given a point $x \in X$ and a point $y \in \rho^{-1}(x)$, we know that the
local ring $\cO_{y,\tX}$ is a finite Cohen--Macaulay extension of the
Gorenstein local ring $\cO_{x,X}$.  According
to~\cite[Theorem~3.3.7]{bh:cm}, $\cO_{y,\tX}$ is Gorenstein if and only if
\begin{equation}\label{eqn:tX-loc-gor}
\cO_{y,\tX} \simeq \Hom_{\cO_{x,X}}(\cO_{y,\tX}, \cO_{x,X}).
\end{equation}
Consider the fact that
\[
i^*_x\cF = i^*_x \rho_*\cO_{\tX} \simeq \bigoplus_{y \in \rho^{-1}(x)}
\cO_{y,\tX}.
\]
Obviously,~\eqref{eqn:tX-loc-gor} implies that
\begin{equation}\label{eqn:X-loc-gor}
i^*_x \cF \simeq \Hom_{\cO_{x,X}}(i^*_x\cF, \cO_{x,X}).
\end{equation}
Conversely, if~\eqref{eqn:X-loc-gor} holds, then by considering the action
of each $\cO_{y,\tX}$ on each side of this isomorphism, we see
that~\eqref{eqn:tX-loc-gor} must hold as well.  Thus,~\eqref{eqn:tX-loc-gor} and~\eqref{eqn:X-loc-gor} are equivalent.  On
the other hand, by~\cite[Proposition~7.24(iii)]{peskine} (for instance),
\[
\Hom_{\cO_{x,X}}(i^*_x\cF, \cO_{x,X}) \simeq i^*_x \cHom(\cF,\cO_X)
\simeq i^*_x (\D\cF).
\]
Now,~\eqref{eqn:X-loc-gor}
is true for all $x$ because it is equivalent to the statement that
$i^*_x\cF \simeq i^*_x\D\cF$.  Therefore,~\eqref{eqn:tX-loc-gor} is true
for all $y$, so $\tX$ is Gorenstein.
\end{proof}

We conclude this section with the statement of a purely ring-theoretic
version of the preceding result.  The authors are not aware of a direct
proof of this statement in the setting of commutative algebra.  Note that the implicit hypothesis that $X$ satisfies condition~(Q) is not needed here because only affine schemes are involved.

\begin{cor}\label{cor:comm-alg}
  Let $A$ be a Gorenstein domain.  Let $K$ be a finite extension of
  the fraction field of $A$, and let $B$ be the integral closure of
  $A$ in $K$.  Let $I \subset A$ be a radical ideal of codimension at
  least $2$, and let $T$ be a set of generators for $I$.
\begin{enumerate}
\item If $B_f$ is Cohen--Macaulay for all $f \in T$, then $B$ is
  Cohen--Macaulay.  (Equivalently, $B$ is Cohen--Macaulay if $B_\cP$
  is for all prime (resp. maximal) ideals $\cP$ of $B$ lying over
  prime (resp. maximal) ideals of $A$ not containing $I$.)
\item If $B_f \simeq \Hom_{A_f}(B_f,A_f)$ for all $f \in T$, then $B$ is Gorenstein.
\end{enumerate}
\end{cor}

\section{Perverse Constructible Sheaves}
\label{sect:icc}

We assume henceforth that $X$ and $G$ are separated schemes over
$S=\Spec k$ for some field $k$ and that $U$ is smooth. In this
section, we establish some results on ordinary (constructible)
perverse sheaves on $X$ which we will need in studying special pieces.

Fix a prime number $\ell$ different from the characteristic of $k$, and
let $D(X)$ be the bounded $G$-equivariant derived category of 
constructible $\Qlb$-sheaves on $X$ (in the sense of Bernstein--Lunts).  By an abuse of notation, we use $\D$ to denote the Verdier duality functor in this category: here $\D  = \cRHom(\cdot, a^!\Qlb)$, where $a: X \to \Spec k$ is the structure morphism.

Let $(D(X)^{\le 0}, D(X)^{\ge 0})$ be the perverse $t$-structure on $D(X)$
with respect to the middle perversity:
\begin{align*}
  D(X)^{\le 0} &= \{ F \in D(X) \mid \dim \supp H^{-i}(F) \le i \}, \\
  D(X)^{\ge 0} &= \{ F \in D(X) \mid \dim \supp H^{-i}(\D F) \le i
  \}.
\end{align*}
Let $M(X)$ be the heart of this $t$-structure.  There is an intermediate extension functor $M(U) \to M(X)$.  Given an equivariant local
system $E$ on $U$, we denote by $\IC(X,E)$ the object of $D(X)$ such
that $\IC(X,E)[\dim X] \in M(X)$ is the intermediate extension of $E[\dim
X] \in
M(U)$.

In addition, let $(\std D(X)^{\le0}, \std D(X)^{\ge 0})$ denote the
standard $t$-structure on $D(X)$.  Note that $\std D(X)^{\le -\dim X}
\subset D(X)^{\le 0}$ and $\std D(X)^{\ge -\dim X} \supset D(X)^{\ge
  0}$.

\begin{prop}\label{prop:ratl-smooth}
If $X$ is irreducible and $\IC(X,\Qlb)$ is a sheaf, then in fact
we have $\IC(X,\Qlb) \simeq \Qlb$ (i.e., $X$ is rationally smooth).
\end{prop}
\begin{proof}
Recall that there is a distinguished triangle
\[
\IC(X,\Qlb)[\dim X] \to Rj_*\Qlb[\dim X] \to F \to \IC(X,\Qlb)[\dim X+1],
\]
where $F$ is supported on $Z$, and $F|_Z$ lies in $D(Z)^{\ge 0}$.  In
particular, this implies that $F \in \std D(X)^{\ge - \dim Z}$.
Taking the long exact sequence cohomology sequence associated to the
above distinguished triangle, we see that $H^k(\IC(X,\Qlb)[\dim X]) \simeq
H^k(Rj_*\Qlb[\dim X])$ for all $k < -\dim Z$.  If we take $k = -\dim X$, we find that
$H^0(\IC(X,\Qlb)) \simeq H^0(Rj_*\Qlb)$.

Since $\IC(X,\Qlb)$ is assumed to be a sheaf, we have
$H^0(\IC(X,\Qlb)) \simeq \IC(X,\Qlb)$.  On the other hand, we have
$H^0(Rj_*\Qlb) \simeq j_*\Qlb \simeq \Qlb$, where the last isomorphism
holds because $X$ is assumed to be irreducible.
\end{proof}

\begin{prop}\label{prop:fin-ratl-smooth}
  Let $f: Y \to X$ be a finite morphism of irreducible varieties.  Let
  $V = f^{-1}(U)$, and assume that $f_*(\Qlb|_V)$ is a local system on
  $U$.  If $\IC(X, f_*(\Qlb|_V))$ is a sheaf, then $Y$ is rationally
  smooth.
\end{prop}
\begin{proof}
Since $f$ is finite (and hence affine and proper), $f_*$ is exact and
$t$-exact by~\cite[Corollaire~2.2.6]{bbd}, and in particular,
$f_*\IC(Y,\Qlb)$ is an intersection cohomology complex on $X$, namely,
it is $\IC(X, f_*(\Qlb|_V))$.  This complex is, by assumption, actually a
sheaf.  Now, $f_*$ kills no nonzero sheaf, so the fact that $f_*\IC(Y,
\Qlb)$ is a sheaf implies that $\IC(Y,\Qlb)$ itself is a sheaf.  By
Proposition~\ref{prop:ratl-smooth}, $Y$ is rationally smooth.
\end{proof}

Since the morphism obtained by $S_2$-extension of a finite morphism is
also finite, the same argument as above gives us the following result
relating intersection cohomology complexes on a scheme obtained by
$S_2$-extension with those on the original scheme.  This fact will be a
vital step in the calculations of Section~\ref{sect:tP}, as anticipated by
Lusztig in his original formulation of
Conjecture~\ref{conj:lusztig}~\cite[\S 0.4]{lus:notes}.

\begin{prop}\label{prop:S2-IC}
Let $\rho: \tX \to X$ be the $S_2$-extension of a finite morphism $\rho_1:
\tU \to U \subset X$.  Let $E$ be a local system on $\tU$, and assume that
$\rho_*E$ is a local system on $U$.  Then we have $\rho_*\IC(\tX, E) \simeq
\IC(X, \rho_*E)$. \qed
\end{prop}

We close this section with the following result expressing the size of
fibers of the normalization map in terms of intersection cohomology.
\begin{prop}\label{prop:normfiber} Let $X$ be an irreducible variety with
  rationally smooth normalization $\bar{X}$, and let
  $\nu:\bar{X}\to X$ be the normalization morphism.  Then for any
  $x\in X$, $|\nu^{-1}(x)|=\dim H^0_x(\IC(X,\Qlb))$.  If $X$ is also
  rationally smooth, then $X$ is unibranch.
\end{prop}
\begin{proof} Since $\nu$ is a finite morphism, it is exact and
  $t$-exact.  This and the fact that $\nu$ is birational imply that
  $\nu_*\Qlb\simeq \nu_*\IC(\bar{X},\Qlb)\simeq
  \IC(X,\nu_*\Qlb)\simeq \IC(X,\Qlb)$.  Taking stalks at $x$ gives
  $\nu_*(\Qlb|_{\nu^{-1}(x)})\simeq \IC(X,\Qlb)_x$, and hence,
  $H^0_x(\IC(X,\Qlb))\simeq \Qlb^{|\nu^{-1}(x)|}$.  The formula for the
  fiber size follows by taking dimensions.  Finally, if $X$ is
  rationally smooth, then $\IC(X,\Qlb)\simeq \Qlb$, so
  $|\nu^{-1}(x)|=1$.
\end{proof}

\begin{rmk} This proposition is known, but we have provided a proof for lack
  of a suitable reference.  The statement (without the assumption that
  $\bar{X}$ is rationally smooth) is given without proof in
  \cite[5E]{bs:green}.
\end{rmk}

\section{The Geometry of $\tP$}
\label{sect:tP}

In this section, we prove Theorem~\ref{thm:main}.  The field $k$ is
now assumed to be algebraically closed of good characteristic for the
group $G$.  We begin by observing that $\cIC(P, \rho_{1*}\cO_{\tC_1})$
is defined.  Indeed, $\tC_1$ is open dense in $P$, and its complement
has codimension at least 2.  Also, the stalk of $\rho_{1*}\cO_{\tC_1}$
at $x\in C_1$ is just the direct sum of $|F|$ copies of $\cO_{C_1,x}$.
Since $C_1$ is smooth, $\rho_{1*}\cO_{\tC_1}$ certainly satisfies the
condition of Lemma~\ref{lem:F-IC}.  We may thus define $\tP$ by
$S_2$-extension:
\[
\tP = \bSpec \cIC(P, \rho_{1*}\cO_{\tC_1}).
\]
Now, $\tC_1$ is regular and $S_2$ (because it is smooth), and its
complement in $\tP$ (which has codimension at least $2$) is $S_2$.  By
Serre's criterion, $\tP$ is normal.  The first part of
Theorem~\ref{thm:main} is then immediate from
Theorem~\ref{thm:S2-ext}.

The remainder of Theorem~\ref{thm:main} is given by
Propositions~\ref{prop:tP-Gorenstein}--\ref{prop:tP-strata} below.

\begin{prop}\label{prop:tP-Gorenstein} \begin{enumerate} \item The variety $\tP$ is endowed
    with natural actions of $F$ and $G$, and these actions commute.
    If we regard $F$ as acting trivially on $P$, then $\rho$ is both
    $G$ and $F$-equivariant.  
\item The variety $\tP$ is rationally
    smooth.  Moreover, if $\operatorname{char} k=0$, then $\tP$ is
    Gorenstein.
  \end{enumerate}
\end{prop}
\begin{proof}

  First, note that $\rho_{1*}\cO_{\tC_1}$ naturally has the structure
  of a $(G \times F)$-equivariant sheaf (where $F$ acts trivially on
  $P$), so $\cIC(P, \rho_{1*}\cO_{\tC_1})$ acquires one as well.
  Applying the $\bSpec$ construction to an equivariant sheaf produces
  a scheme carrying an group action and an equivariant morphism.

Next, the rational smoothness of $\tP$ follows from
Proposition~\ref{prop:fin-ratl-smooth}, together with the fact that
$\IC(P,E)$ is a sheaf for any equivariant local system $E$ on $C_1$. 
(See~\cite[Proposition~0.7(c)]{lus:notes}.)

Finally, observe that because $\tP$ is normal, the canonical morphism $\rho: \tP \to P$ factors
through the normalization $\bar P$ of $P$: 
\[
\xymatrix{
\tP \ar[r]^{\bar\rho} \ar@/_1pc/[rr]_{\tpi} & 
\bar P \ar[r]^{\nu} & P}
\]
By invoking Proposition~\ref{prop:Spec-IC}, we see that $\tP$ can also be
constructed as 
\[
\bSpec \cIC(\bar P, (\bar\rho|_{\tC_1})_*\cO_{\tC_1}).
\]
Now, $\bar P$ is Gorenstein in characteristic $0$ by the theorem of
Hinich--Panyushev~\cite{hinich, panyushev}, so $\tP$ is Gorenstein as
well by Corollary~\ref{cor:S2-Gor}.  (Note that $\bar P$ satisfies
condition~(Q), since it is normal and quasiaffine.)
\end{proof}

\begin{prop}\label{prop:unibranch}
Each special piece $P$ is unibranch.
\end{prop}
\begin{proof} This follows from Proposition~\ref{prop:normfiber}, the
  preceding result, and the fact that $P$ is rationally smooth.
\end{proof}

\begin{prop}\label{prop:tP-fibers}
The morphism $\bar\rho:\tP\to\bar P$ is the algebraic quotient of
the $F$-action while $\rho:\tP\to P$ is the topological quotient.  In
particular, $F$ acts transitively on the fibers of $\rho$.
\end{prop}
\begin{proof}
Since $P$ is unibranch, $P$ is homeomorphic to $\bar P$, and it
suffices to show that $\bar P\simeq \tP/F$.

The functor $\IC$ is an equivalence of categories between appropriate
categories of sheaves on $C_1$ and $P$, and it accordingly preserves
finite limits.  In particular, it preserves $F$-fixed objects, so
$\cIC(P, \rho_{1*}\cO_{\tC_1})^F\simeq
\cIC(P,(\rho_{1*}\cO_{\tC_1})^F)\simeq \cIC(P,\cO_{C_1})$.  The result
now follows, since we have $\tP/F\simeq\bSpec \cIC(P, \rho_{1*}\cO_{\tC_1})^F$ and
$\bar P\simeq \bSpec \cIC(P,\cO_{C_1})$.
\end{proof}

\begin{prop}\label{prop:tP-strata}
For each parabolic subgroup $H \subset F$, we have $\rho^{-1}(C_H) \simeq
\tC_H$.
\end{prop}

We prove this proposition in two steps.  First, in
Lemma~\ref{lem:tP-strata-K}, we obtain a general description of the
varieties $\rho^{-1}(C_H)$ in terms of unknown $F$-stabilizers.  This
description will suffice to prove the proposition when $F$ is abelian,
and in particular, for the classical groups.  Then, in
Lemma~\ref{lem:tP-strata-H}, we show by case-by-case considerations
that the $F$-stabilizers of points in $\rho^{-1}(C_H)$ are in fact
conjugates of $H$ for the exceptional groups.

\begin{lem}\label{lem:tP-strata-K}
  Let $H$ be a parabolic subgroup of $F$.  Each connected component of
  $\rho^{-1}(C_H)$ is isomorphic to $(\tC_H)^\circ$.  Let $K_H$ be the
  stabilizer in $F$ of some closed point of $\rho^{-1}(C_H)$.  Then
  $|K_H| = |H|$, and there is a subgroup $L_H \subset N_F(K_H)$ such
  that $|L_H| = |N_F(H)|$ and $\rho^{-1}(C_H) \simeq (\tC_H)^\circ
  \times_{L_H} F$.  Moreover, if $H$ is conjugate to a subgroup of
  another parabolic $H'$, then $K_H$ is conjugate to a subgroup of
  $K_{H'}$.
\end{lem}
\begin{proof}
Let $E$ denote either the regular representation of $F$ or, by abuse of
notation, the corresponding local system on $C_1$.  We will calculate
$\IC(P,E)|_{C_H}$ in a way that reflects the structure of $\rho^{-1}(C_H)$
and then compare with the known calculations following~\cite{lus:notes} to
prove the result. 

Consider the commutative diagram
\[
\xymatrix{
\;\rho^{-1}(C_H)\; \ar@{^{(}->}[r]^-{\ti_H}\ar[d]_{\rho} & \tP \ar[d]^{\rho} \\
\;C_H\; \ar@{^{(}->}[r]_{i_H} & P}
\]
From Proposition~\ref{prop:S2-IC}, we have $\rho_*\IC(\tP,\Qlb) \simeq
\IC(P, \rho_*\Qlb|_{C_1})$. Moreover, because $\rho|_{\rho^{-1}(C_1)}$
is a principal $F$-bundle, $\rho_*\Qlb|_{C_1} = E$.  On the other
hand, $\IC(\tP,\Qlb) \simeq \Qlb$ as $\tP$ is rationally smooth, so we
have
\[
\rho_*\Qlb \simeq \IC(P,E),
\]
and hence $\IC(P,E)|_{C_H} \simeq (\rho_*\Qlb)|_{C_H}$.  Now, since $\rho$
is proper, we know that $\rho_*\Qlb|_{C_H} \simeq
\rho_*(\Qlb|_{\rho^{-1}(C_H)})$. We seek to
understand $\rho_*(\Qlb|_{\rho^{-1}(C_H)})$.

Choose a point $x \in C_H$ and a point $\tilde x \in \rho^{-1}(x)$.  Since
the map $\rho: \rho^{-1}(C_H) \to C_H$ is finite and $G$-equivariant, the
stabilizer in $G$ of $\tilde x$, which we denote $G^{\tilde x}$, must be a
finite-index subgroup of the stabilizer $G^x$ of $x$.  The connected
component of $\rho^{-1}(C_H)$ containing $\tilde x$, which will be
denoted $B$, must be isomorphic to the homogeneous space $G/G^{\tilde x}$.
Then, since $F$ acts transitively on the fiber $\rho^{-1}(x)$, and the
actions of $F$ and $G$ commute, it follows that every connected component
of $\rho^{-1}(C_H)$ is isomorphic to $G/G^{\tilde x}$.  Let $L_H$ be the
subgroup of $F$ that preserves $B$ (without necessarily fixing $\tilde
x$).  The preceding discussion shows that $\rho^{-1}(C_H)$ is isomorphic to
$B \times_{L_H} F$ (where $a \in F$ acts on a pair $(b,f) \in B
\times_{L_H} F$ by $a \cdot (b,f) = (b, fa^{-1})$).  In particular, the
number of connected components of $\rho^{-1}(C_H)$ is $[F : L_H]$.  

Let $K_H$ be the stabilizer in $F$ of $\tilde x$.  Since the actions of $F$
and $G$ commute, it follows that $K_H$ is also the $F$-stabilizer of every
other point in $B$.  This implies that $K_H$ is a normal subgroup of $L_H$.
 Now, the group $L_H/K_H$ acts simply transitively on $\rho^{-1}(x)
\cap B$, so this is the group of deck transformations of $B$ over $C_H$. 
Let $A'(C_H) = L_H/K_H$.  We also have $A'(C_H) \simeq G^x/G^{\tilde x}$,
which is the quotient of $A(C_H) \simeq G^x/(G^x)^\circ$ by $G^{\tilde
x}/(G^x)^\circ$. 

The local system $(\rho|_B)_*\Qlb$ on $C_H$ corresponds to the regular
representation of $A'(C_H)$, and the full local system
$\rho_*(\Qlb|_{\rho^{-1}(C_H)})$ is then clearly just the direct sum
of $[F:L_H]$ copies of the regular representation of $A'(C_H)$.  It is
easily checked that the action of $L_H/K_H$ on the space $E^{K_H}$ of
$K_H$-invariant vectors in $E$ is also the direct sum of $[F:L_H]$
copies of its regular representation.  Thus, $\IC(P,E)|_{C_H} \simeq
E^{K_H}$ as an $A(C_H)$-representation.  Set $Q\subset A(C_H)$ equal
to the kernel of this representation.  Since $E^{K_H}$ is a faithful
representation of $A'(C_H)$, we obtain $A'(C_H)\simeq A(C_H)/Q$ as
groups and left $A(C_H)$-spaces.  (In other words, if $A'(C_H)$ is
viewed as a quotient of $A(C_H)$, then $A'(C_H)=A(C_H)/Q$.)

On the other hand, according to~\cite[Proposition~0.7]{lus:notes},
$\IC(P,E)|_{C_H}$ is the representation of
$A''(C_H)\overset{\mathrm{def}}{=} N_F(H)/H$ on $E^H$.
Following~\cite{lus:notes, as}, this group is a direct factor of $\bar
A(C_H)$ and hence naturally a quotient of $A(C_H)$.  The same argument
now shows that $A''(C_H)\simeq A(C_H)/Q$ and hence $A'(C_H) \simeq
A''(C_H)$ as groups and $A(C_H)$-spaces.  Moreover, $E^H$ and
$E^{K_H}$ are isomorphic representations via this isomorphism.  In
particular, we have $|K_H| = |H|$, since $\dim E^{K_H} = [F:K_H]$ and
$\dim E^H = [F:H]$.  It follows immediately that $|L_H| = |N_F(H)|$.
(However, we cannot conclude that $K_H$ is conjugate to $H$; see
Remark~\ref{rmk:failure}.)

The fact that $A'(C_H) \simeq A''(C_H)$ as homogeneous spaces implies
that $G^{\tilde x}$ is precisely the kernel $G^x_F$ of the canonical
map $G^x \to N_F(H)/H$.  Thus, $B \simeq G^x/G^x_F = (\tC_H)^\circ$,
and $\rho^{-1}(C_H) \simeq (\tC_H)^\circ \times_{L_H} F$.

Finally, the points of $\tP$ fixed by $K_H$ form a closed subvariety.  If
we repeat the above argument with another parabolic $H'$ with $H \subset
H'$, so that $C_{H'} \subset \overline{C_H}$, we see that $K_H$ must be
contained in the $F$-stabilizers of points of $\overline B \cap
\rho^{-1}(C_{H'})$.  Every such stabilizer is conjugate to $K_{H'}$, so
$K_H$ is conjugate to a subgroup of $K_{H'}$.
\end{proof}

If $F$ is abelian, then $|L_H| = |N_F(H)|$ implies that both groups
are in fact equal to $F$.  Thus, $\rho^{-1}(C_H) \simeq (\tC_H)^\circ
\times_{N_F(H)} F=\tC_H$.  

It remains to identify the $K_H$'s and $L_H$'s for the exceptional
groups.  There, the only nontrivial groups $F$ that occur are
symmetric groups $\fS_n$ with $2 \le n \le 5$.  The following lemma
gives us the required information about the $K_H$'s.

\begin{lem}\label{lem:tP-strata-H}
Let $F = \fS_n$ with $2 \le n \le 5$.  Let $\{K_H\}$ be a collection of
subgroups of $F$, where $H$ ranges over the parabolic subgroups of $F$. 
Assume that $|K_H| = |H|$ and that $K_{H_1}$ is conjugate to a subgroup of
$K_{H_2}$ whenever $H_1$ is conjugate to a subgroup of $H_2$.  Then each
$K_H$ is conjugate to $H$.
\end{lem}
\begin{proof}
If $F = \fS_2$ there are no nontrivial cases of $H$ to consider.

If $F = \fS_3$, we must consider $H = \fS_2$.  It is clear that every
subgroup of $F$ of order $2$ is conjugate to $H$.

If $F = \fS_4$, then the nontrivial possibilities for $H$ are $\fS_2$,
$\fS_3$, and $\fS_2 \times \fS_2$.  The last one is a Sylow $2$-subgroup of
$F$, so every subgroup of order $4$ is conjugate to it.  Next, it is easy to
verify by hand calculation that every subgroup of $\fS_4$ generated by an
element of order $3$ and another of order $2$ either has more than $6$
elements or is conjugate to $\fS_3$.  Finally, if $H = \fS_2$, we now know
that $K_H$ must be conjugate to a subgroup of $\fS_3$, so $K_H$ is
conjugate to $\fS_2$ by the preceding paragraph. 

If $f = \fS_5$, there are five nontrivial parabolic subgroups up to
conjugacy.  Another hand calculation shows that any subgroup generated by
an element of order $4$ and another of order $2$ either has size different
from $24$ or is conjugate to $\fS_4$.  Next, if $H$ is any of $\fS_3$,
$\fS_2 \times \fS_2$, or $\fS_2$, then $K_H$ must be conjugate to a
subgroup of $\fS_4$, so by the previous paragraph, $K_H$ is conjugate to
$H$.  Finally, suppose $H = \fS_3 \times \fS_2$.  Then $K_H$ must contain a
subgroup conjugate to $\fS_3$.  Again, an easy calculation shows that every
subgroup generated by $\fS_3$ and an element of order $2$ either has size
different from $12$ or is conjugate to $\fS_3 \times \fS_2$. 
\end{proof}

\begin{rmk}\label{rmk:failure}
The above lemma does not hold in general for $F$ an finite Coxeter group. 
For example, suppose $F$ is the Weyl group of type $B_2$, generated by
simple reflections $s$ and $t$ with $(st)^4 = 1$.  The groups $\langle s
\rangle$ and $\langle t \rangle$ are representatives of the two conjugacy
classes of nontrivial parabolic subgroups.  If we set $K_{\langle s
\rangle} = \langle s \rangle$ and $K_{\langle t \rangle} = \langle (st)^2
\rangle$, the hypotheses of the lemma are satisfied, but evidently
$K_{\langle t \rangle}$ is not conjugate to $\langle t \rangle$.
\end{rmk}

We now know that $K_H$ is conjugate to $H$ in all cases.  Returning to the
setting of Lemma~\ref{lem:tP-strata-K}, we see that since $L_H \subset
N_F(K_H) \simeq N_F(H)$ and $|L_H| = |N_F(H)|$, we must in fact have $L_H =
N_F(K_H)$, so $\rho^{-1}(C_H) \simeq (\tC_H)^\circ \times_{N_F(H)} F =
\tC_H$.  The proof of Proposition~\ref{prop:tP-strata} is now complete, and
hence, so is the proof of Theorem~\ref{thm:main}. 

Finally, we observe that any smooth variety containing
$\tC_1$ as a dense open set together with a finite morphism to $P$
extending $\rho_1$ must  coincide with $\tP$.  In particular:
\begin{cor} If $\tchar k=0$ and $G$ is classical, $\tP$ is
  isomorphic to the smooth variety over $P$ constructed by Kraft--Procesi.
\end{cor}
\begin{proof}  This follows immediately from the theorem and
  Proposition~\ref{prop:unique}.
\end{proof}

\section{Normality of Special Pieces}
\label{sect:normal}

Recall that Conjecture~\ref{conj:lusztig} contains implicitly the
additional Conjecture~\ref{conj:normal} that all special pieces are normal.  In the
classical types in characteristic $0$, this statement follows from the
work of Kraft--Procesi~\cite{kp:special}; they show that each special
piece $P$ is the algebraic quotient of $\tP$ by $F$, so by
Proposition~\ref{prop:tP-fibers}, $P$ is normal.

In positive characteristic or for the exceptional types in
characteristic $0$, there is no uniform answer.  Of course, some
special pieces consist only of a single unipotent class, so those ones
are obviously normal (and even smooth).  In other cases, it is
known that the full closure of a special unipotent class in the
unipotent variety is normal.  Since a special piece is an open
subvariety of its closure, the normality of the closure implies the
normality of the special piece.  Normality of closures of unipotent
classes (or, more typically, nilpotent orbits) has been studied
extensively by a number of authors, so this technique gives
information about a large number of special pieces.  In this section,
we list the normality results that can be obtained in this way.

The following proposition summarizes the situation for classical groups.

\begin{prop}
Let $G$ be a simple algebraic group of classical type over an algebraically closed field $k$ of good characteristic.  Let $C_1$ be a special unipotent class, and let $P$ be the corresponding special piece.
\begin{enumerate}
\item If $\tchar k = 0$ or $G$ is of type $A_n$, then $P$ is normal.
\item $P = C_1$ if and only if $\bar A(C_1) = 1$.  In that case, of course, $P$ is normal.\label{it:triv}
\item If $G$ is of type $B_n$ and $C_1$ is the subregular class, then $P$ is normal.
\end{enumerate}
\end{prop}
We remark that it is easy to determine whether $\bar A(C_1) = 1$ for a given special class, using the straightforward combinatorial descriptions of that group given in, say,~\cite{as} or~\cite{lus:notes}.
\begin{proof}
As we remarked above, in characteristic $0$, the result follows from the work of Kraft--Procesi~\cite{kp:special}.  In type $A_n$, every unipotent class is special, so every special piece consists of a single class.

Next, it is obvious that $P = C_1$ if $\bar A(C_1) = 1$; the other implication follows from~\cite[Theorem~2.1]{as}.

Finally, the subregular classes in type $B_n$ occur in Thomsen's list~\cite[\S 9]{thomsen} of classes known to have normal closure in any good characteristic.
\end{proof}

\begin{rmk}
Thomsen lists many more classes with normal closures in the classical types, including the subregular class in all types, but it happens that all other classes listed by him fall into case~\eqref{it:triv} of the proposition above.
\end{rmk}

\begin{table}
\begin{center}
\def\rr{\raggedright}
\begin{tabular}{|l|p{1.5in}|l|p{0.5in}|p{0.74in}|p{0.7in}@{}l|}
\rotatebox{90}{\hbox{Group}} & Smooth & 
Normal & 
\parbox[b]{0.5in}{\centering Normal in\break char. $0$} &
\parbox[b]{0.74in}{\centering Normal if BPS conj. is true} &
Unknown & \\
\hline
$G_2$ &
 $G_2$, $1$ & $G_2(a_1)$ & & & & \\
\hline
$F_4$ &
 \rr $F_4$, $F_4(a_1)$, $F_4(a_2)$, $C_3$, $B_3$, $\tilde A_2$, $A_2$,
$A_1+\tilde A_1$, $1$ & 
 & \rr $F_4(a_3)$, $\tilde A_1$ & & & \\
\hline
$E_6$ &
 \rr $E_6$, $E_6(a_1)$, $D_5$, $D_5(a_1)$, $A_4 + A_1$, $D_4$, $A_4$,
$A_3$, $A_2+2A_1$, $2A_2$, $A_2+A_1$, $2A_1$, $A_1$, $1$ &
 $E_6(a_3)$ & \rr $D_4(a_1)$, $A_2$ & & & \\
\hline
$E_7$ &
  \rr $E_7$, $E_7(a_1)$, $E_7(a_2)$, $E_6$, $E_6(a_1)$, $E_7(a_4)$,
$D_6(a_1)$, $D_5+A_1$, $A_6$, $D_5$, $D_5(a_1)+A_1$, $A_4+A_2$, $A_4+A_1$,
$(A_5)''$, $A_3+A_2+A_1$, $A_4$, $A_3+A_2$, $D_4$, $(A_3+A_1)''$,
$A_2+3A_1$, $2A_2$, $A_3$, $A_2+2A_1$, $(3A_1)''$, $2A_1$, $A_1$, $1$ & 
  $E_7(a_3)$ & &
  \rr $E_7(a_5)$, $E_6(a_3)$, $D_5(a_1)$,
  $D_4(a_1)$, $A_2+A_1$, $A_2$ &
  $D_4(a_1)+A_1$ & \\
\hline
$E_8$ &
  \rr $E_8$, $E_8(a_1)$, $E_8(a_2)$, $E_8(a_4)$, $E_8(b_4)$, $E_7(a_1)$,
$E_8(a_6)$, $D_7(a_1)$, $E_6(a_1)+A_1$, $D_7(a_2)$, $E_6$, $D_5+A_2$,
$E_6(a_1)$, $E_7(a_4)$, $A_6+A_1$, $A_6$, $D_5$, $D_4+A_2$, $A_4+A_2+A_1$,
$D_5(a_1)+A_1$, $A_4+A_2$, $A_4+A_1$, $A_4$, $A_3+A_2$, $D_4$, $A_3$,
$A_2+2A_1$, $2A_1$, $A_1$, $1$ &
  $E_8(a_3)$ & &
  \rr $E_8(a_5)$, $E_8(b_5)$, $E_8(b_6)$, $E_8(a_7)$, $A_4+2A_1$,
  $D_4(a_1)+A_2$, $D_4(a_1)+A_1$, $D_4(a_1)$, $2A_2$, $A_2+A_1$, $A_2$ &
  \rr $E_7(a_3)$, $D_6(a_1)$, $E_6(a_3)$, $D_5(a_1)$ & \\
\hline
\end{tabular}
\end{center}
\caption{Normality of special pieces in the exceptional types}
\label{tbl:exc}
\end{table}

In Table~\ref{tbl:exc}, we indicate what is known for special pieces
in the exceptional groups.  We name a special piece by giving the
Bala--Carter label of the special class it contains.  The column
labelled ``Smooth'' lists all special pieces that contain only a
single class (this is easily deduced from, say, the partial order
diagram of unipotent classes in~\cite[Chapter~13]{carter}).  Among the
remaining special pieces, those with normal closure in any good
characteristic (following Thomsen~\cite{thomsen}) are listed in the
next column, and those known to have normal closure only in
characteristic $0$ (following Broer~\cite{broer:f4} and
Sommers~\cite{sommers:e6}) appear in the column after that.

Before explaining the last two columns, we remark that in types $E_7$
and $E_8$, the normality question has not been answered for all
nilpotent orbit closures, even in characteristic $0$.  However, a
number of specific orbits are known to have nonnormal closures, and
Broer, together with Panyushev and Sommers, has conjectured that all
remaining orbits have normal closures (see the Remarks at the end
of~\cite[\S 7.8]{broer:decomp}).  Sommers has verified this conjecture
in a large number of cases~\cite{sommers:pc}.  If the
Broer--Panyushev--Sommers conjecture is true, it will imply the
normality of a number of special pieces, listed in the penultimate
column.

Finally, the last column lists special pieces whose closures are known
to be nonnormal.  To establish normality for these special pieces,
some additional technique will be required.

The following proposition summarizes the information that can found in
the table.

\begin{prop}
  Let $G$ be a simple algebraic group of exceptional type over an
  algebraically closed field $k$ of good characteristic.
\begin{enumerate}
\item If $G$ is of type $G_2$, all special pieces are normal.
\item If $G$ is of type $F_4$ or $E_6$ and $\tchar k = 0$, all special
  pieces are normal.  If $\tchar k > 0$, all but two special pieces
  are known to be normal.
\item If $G$ is of type $E_7$ (resp.~$E_8$), then all but seven
  (resp.~fifteen) special pieces are known to be normal.  If $\tchar k
  = 0$ and the Broer--Panyushev--Sommers conjecture holds, then all
  but one (resp.~four) special pieces will be known to be normal. \qed
\end{enumerate}
\end{prop}

\bibliographystyle{pnaplain}

\end{document}